\documentclass[12pt,leqno]{amsart}

\usepackage{amsthm,amsfonts,amssymb,amsmath,oldgerm}
\usepackage{epsfig}
\numberwithin{equation}{section}
\usepackage{graphicx}

\newtheorem{theorem}{Theorem}[section]

\newtheorem{lemma}{Lemma}[section]
\newtheorem{definition}{Definition}[section]

\newtheorem{remark}{Remark}[section]
\newcommand{\f}{\frac}
\newcommand{\e}{\epsilon}
\newcommand{\be}{\begin{equation}}
\newcommand{\ee}{\end{equation}}
\newcommand{\bs}{\begin{split}}
\newcommand{\es}{\end{split}}
\newcommand{\bc}{\begin{cases}}
\newcommand{\ec}{\end{cases}}

\newcommand{\pa}{\partial}

\begin{document}

\title[Singularly perturbed heat equation in a circle]{Numerical approximation of the singularly perturbed heat equation in a circle}

\author[Youngjoon Hong]{Youngjoon Hong}
\address[Youngjoon Hong]{Institute for Scientific Computing and Applied Mathematics, Indiana University, Bloomington, Indiana, USA.}
\email{hongy@indiana.edu}


\date{\today}

\begin{abstract}
In this article we study the two dimensional singularly perturbed heat equation in a circular domain.
The aim is to develop a numerical method with a uniform mesh, avoiding mesh refinement at the boundary thanks to the use of a relatively simple representation of the boundary layer.
We provide the asymptotic expansion of the solution at first order and derive the boundary layer element resulting from the boundary layer analysis.
We then perform the convergence analysis introducing the boundary layer element in the finite element space thus obtaining what is called an ``enriched Galerkin space''.
Finally we present and comment on numerical simulations using a quasi-uniform grid and the modified finite element method.
\end{abstract}
\maketitle

\noindent{\bf}\\

\section{\bf Introduction}
In this article we consider the two-dimensional singularly perturbed heat equation of the form
\be
	\begin{cases}
	\dfrac{\partial u^\e}{\partial t} - \e \Delta u^\e = f, \quad \text{in  }D \times (0,T),\\
	u^\e(x,y,t) = 0, \quad \text{on  } \pa D \times (0,T),\\
	u^\e(x,y,0) = u_0(x,y), \quad \text{on  } D, 
	\end{cases} \label{eq_main}
\ee
where $0<\e \ll 1$ is the heat conductivity and $D$ is the unit disc centered at $(0,0)$. The functions $f=f(x,y,t)$ and $u_0=u_0(x,y)$ are assumed to be sufficiently regular. We also assume the compatibility condition
\be 
	u_0 = 0 \quad \text{on  } \pa D.
\label{eq:comp}	
\ee

The numerical methods for singularly perturbed problems have been studied in many articles.
In \cite{RST08}, \cite{RU03}, \cite{SM05} and \cite{ST03} the authors proposed numerical methods for stationary convection-diffusion equations using finite element methods.
More recently, one can find numerical results for parabolic type problems in \cite{BCGJ07}, \cite{CG12}, \cite{KS11}, \cite{LM10}, and \cite{VS13}.
In those articles the authors utilized mesh refinement near the boundary.
Furthermore, for the parabolic type cases, the authors mainly focused on the finite difference methods in a unit interval or rectangular domains and this did not address the issue of the curved boundary in the context of time dependent problems.
Our object here is to address these issues. We do so and avoid the costly mesh refinements at the boundary using some results from the boundary layer analysis.

In an earlier work \cite{GHT10}, which gives the theoretical background of this article, one can find the asymptotic expansion for the solution of \eqref{eq_main} in a smooth domain and the $H^1$-estimates for the ``error'' (see below).
In this article devoted to the numerical analysis of \eqref{eq_main} we first look for $H^2$-estimates of the ``error'', which play an important role in the numerical analysis, thus completing the results in \cite{GHT10}; see e.g. Remark \ref{rem1} and Theorem \ref{main_thm1}.
We then introduce the boundary layer element based on the boundary layer analysis.
Incorporating the boundary layer element in the finite element space, we obtain the proposed ``enriched'' Galerkin space to be used in the numerical simulations (together with a ``uniform'' mesh).
Then we perform the convergence analysis applying the Aubin-Nitsche trick (duality argument) as in \cite{AJ67}, \cite{CTW00}, \cite{CJ09} and \cite{NJ70}.
We then present the results of our numerical simulations using a {\it quasi-uniform grid} and the enriched finite element space.\\
For the time-dependent problems, the mesh refinement is rather costly since we have to consider large scale matrices at each time step.
Moreover, if the domain is not rectangular as is the case in our problem, the finite difference methods are not practical.
This justifies the approaches used in this article which, we believe, should also be applicable to many other types of time-dependent singularly perturbed problems such as reaction-diffusion equations.
In addition we intend, in the future, to extend our numerical methods to the linearized Navier-Stokes equations when the viscosity is small; see e.g. \cite{GHT11}, \cite{TW95}, \cite{TW97}, and a forthcoming article \cite{HT14}.

The concept of enriched space and boundary layer element was first introduced in \cite{KS87}. Unaware of \cite{KS87}, the authors of \cite{CT02} and \cite{CTW00} introduced independently a similar concept for the one-dimensional equations.
Especially, in \cite{CTW00}, the authors studied the numerical analysis of the one-dimensional time-dependent problem. In \cite{JT08}, \cite{JT06}, and \cite{JT10} the authors presented the numerical methods for the two-dimensional stationary convection-diffusion equations using the finite element methods and finite volume methods in a rectangular domain.
Lately, in \cite{HJL13}, \cite{HJT13} and \cite{JT14}, the time-independent equations were considered in a circular domain.

This article is organized as follows. In Section \ref{sec2}, we look for $L^\infty(0,T;H^2)$ estimates of the error between the exact solution $u^\e$ and its asymptotic expansion to a certain order.
In Section \ref{sec_fem}, we define the boundary layer element which we incorporate in the finite element space, and provide the convergence analysis for the enriched finite element approximation using the duality argument.
In Section \ref{sec_num}, we present approximate boundary layer elements and perform numerical applications using these elements.

\section{\bf Asymptotic analysis} \label{sec2}
\subsection{Boundary fitted coordinates}
The weak formulation of $\eqref{eq_main}_1$ reads\\
To find $u^\e:(0,T) \longrightarrow H^1_0(\Omega)$ such that
\be
\begin{split}
	   (\pa_t u^{\e},v) +   \e (\nabla u^{\e},\nabla v) 
	   & = (f,v), \quad \forall v \in H^1_0(\Omega),\\
	   u^\e(0) & = u_0.
\end{split}	  
\label{eq_main_var_form}
\ee
The formal limit problem of \eqref{eq_main}, namely when $\e \rightarrow 0$, is easily seen to be
\be
	\begin{cases}
	\dfrac{\partial u^0}{\pa t} = f, \quad \text{in  }D \times (0,T), \\
	u^0(x,y,0) = u_0(x,y) \quad \text{on  } D.
	\end{cases} \label{eq_lim}
\ee
Hence, we easily find the explicit solution $u^0(x,y,t) = u_0(x,y) + \int^t_0 f(x,y,s) ds$.
To investigate the boundary layer in the circular domain, we first introduce the boundary fitted coordinates as in \cite{HJT13}
\begin{equation} \notag
	\begin{cases}
	x = (1-\xi) \cos \eta,\\
	y = (1-\xi) \sin \eta,
	\end{cases}
\end{equation}
where $\xi=1-r$, $r$ is the distance to the center, and $\eta$ is the polar angle from $Ox$.
We then define the domains $D^*$ and $D_{\f{1}{2}}$ as follows:
\be
	\begin{split}
	& D^* = \{(\eta,\xi) \in (0,2 \pi) \times (0,1)  \},\\
	& D_{\f{1}{2}} = \{(\eta,\xi) \in D^*: \xi \leq \f{1}{2}\}.
	\end{split} \label{eq_domain}
\ee
Using this change of variables, we obtain
\be
	  \frac{\partial}{\partial x} = -\cos\eta \frac{\partial}{\partial \xi} - \frac{\sin\eta}{1-\xi} \frac{\partial}{\partial \eta}, \quad 
	  \frac{\partial}{\partial y} = -\sin\eta \frac{\partial}{\partial \xi} + \frac{\cos\eta}{1-\xi} \frac{\partial}{\partial \eta},
	\label{eq_chg}
\ee
and then $\eqref{eq_main}_1$ becomes
\be
	  L_\e(u^\e) = \frac{\pa u^\e}{\pa t} -\epsilon \Delta u^{\epsilon} = 
	  \frac{\pa u^\e}{\pa t} - \frac{\epsilon}{(1-\xi)^2} \frac{\partial ^2 u^{\epsilon}}{\partial \eta^2} + \frac{\epsilon}{1-\xi} \frac{\partial u^{\epsilon}}{\partial \xi} - \epsilon \frac{\partial ^2 u^{\epsilon}}{\partial \xi^2} 
	  = f.
	\label{eq_main_chg}
\ee
\subsection{Convergence analysis}
We first look for the expansion of $u^\e$ at first order:
\be
	u^\e \simeq u^0 + \theta^0,
\ee
where $u^0$ is the solution of $\eqref{eq_lim}_1$ and $\theta^0$ is the first corrector. Setting $f=0$ in \eqref{eq_main_chg}, and using the stretched variable $\xi = \e^{\alpha} \bar \xi$, \eqref{eq_main_chg} is transformed to
\be
	\dfrac{\pa u^\e}{\pa t} - \dfrac{\epsilon}{(1-\e^{\alpha} \bar \xi)^2} \frac{\partial ^2 u^{\epsilon}}{\partial \eta^2} + \dfrac{\epsilon^{1- \alpha}}{1- \e^{\alpha} \bar \xi} \dfrac{\partial u^{\epsilon}}{\partial \bar \xi} - \epsilon^{1- 2 \alpha} \dfrac{\partial ^2 u^{\epsilon}}{\partial \bar \xi^2}=0.
	\label{eq_main_str}
\ee
The dominating terms in \eqref{eq_main_str} are
\be
	\dfrac{\pa u^\e}{\pa t}  - \epsilon^{1- 2 \alpha} \dfrac{\partial ^2 u^{\epsilon}}{\partial \bar \xi^2} = 0,
\label{eq_dom}
\ee
and thus the reasonable thickness of the boundary layer is $\alpha = {\f{1}{2}}$ so that $\xi = \e^{\f{1}{2}} \bar \xi$, with $\bar \xi = O(1)$ in the boundary layer.\\ Then, we obtain the equations for the corrector $\theta^0 = \theta^0(\eta,\xi,t)$:
\be
	\begin{cases}
	\dfrac{\pa \theta^0}{\pa t} - \dfrac{\pa^2 \theta^0}{\pa \bar \xi^2} = 0, \quad \text{in  } D^* \times (0,T),\\
	\theta^0(\eta,0,t) = -u^0, \quad \text{at  } \bar \xi =0,\\
	\theta^0(\eta,\xi,0) = 0,\\
	\theta^0 \longrightarrow 0, \quad \text{as  } \bar \xi \longrightarrow \infty.
	\end{cases} \label{eq_cor}
\ee
The explicit solution $\theta^0$ is
\be
	\theta^0 = -\int^t_0 I(\xi,t-s) \dfrac{\pa u^0}{\pa t}(\eta,0,s) ds,
	\label{sol_cor}
\ee
where 
\be
	\begin{split}
	& I(\xi,t)=\text{erfc}\Big (\dfrac{\xi}{\sqrt{2 \e t}} \Big ),\\
	& \text{erfc}(z) = 1 - \text{erf}(z) = \sqrt{\dfrac{2}{\pi}} \int^{\infty}_{z} \exp \Big (-\dfrac{y^2}{2} \Big ) dy,\\
	& \text{erf}(z) = \sqrt{\dfrac{2}{\pi}} \int^z_0 \exp \Big (-\dfrac{y^2}{2} \Big ) dy;
	\end{split}
\label{eq_err_fuc}	
\ee
see e.g. \cite{Cannon}.
To avoid the singularity of $\theta^0$ at the origin (at $\xi=1$), we introduce the approximation 
\be
	\bar \theta^0(\eta,\xi,t) = \theta^0 \delta(\xi)
\label{e:cut_cor}	
\ee
where $\delta(\xi)$ is a smooth cut-off function such that $\delta(\xi) = 1$ for $0 \leq \xi \leq 1/4$ and $\delta(\xi) = 1$ for $1/2 \leq \xi \leq 1$.\\
We need an additional compatibility condition as in \cite{JPT11} to estimate the higher order derivatives of $\theta_0$; namely
\be
	  \f{\partial u^0}{\partial t}(\eta,0,0) = 0.
\ee
Due to $\eqref{eq_lim}_1$
\be
	  \f{\partial u^0}{\partial t} \Big |_{t=0} = f \Big |_{t=0}, \text{  on  } \partial \Omega,
\ee
and hence we require the compatibility condition
\be
	  f(x,y,0) = 0, \text{  on  } \partial \Omega.
\label{e:comp_cond2}		  
\ee
We now recall the following lemma from \cite{GHT10}.
\begin{lemma} \label{lem_pwbd2}
 For $m=0,1$ and $k \geq 0$, the following pointwise estimates hold for $\theta^0= \theta^0(\eta,\xi,t)$:
 \begin{equation} \label{eq_est1}
	 |\partial^k_\eta \partial^m_{\xi} \theta^0| \leq \kappa \e^{-\f{m}{2}} \exp \Big(-\f{\xi^2}{4 \e t} \Big).
 \end{equation}
 Moreover, for $(\eta,\xi) \in D_{\f{1}{2}}$, and for $j=0$ and $m \geq 2$, and for $j \geq 1$ and $m \geq 0$, we find
 \be \label{eq_est2}
	|\partial^j_t \partial^k_{\eta} \partial^m_{\xi} \theta^0| \leq \kappa \e^{-j-m+\f{1}{2}} \int^t_0 (1+s^{-2j-m+\f{1}{2}}) \exp \Big(-\f{\xi^2}{4 \e s} \Big)ds.
 \ee
 Here and below $\kappa$ is a positive constant which is independent of $\e$ and may depend on the data and which may be different at different occurrences.
\end{lemma}
\noindent
Furthermore, we recall the definition of an $e.s.t.$.
\begin{definition}
A function or a constant depending on $\e$, $\tilde g_\e$, $\tilde g^\e$ is called an exponentially small term, and denoted e.s.t., if there exist $\beta$, $\beta' > 0$ such that for any $k\geq 0$, there exists a constant $c_{\beta,\beta',k} >0$ independent of $\e$ such that
\be
\|\tilde g^\e\|_{H^k} \leq c_{\beta,\beta',k} \exp(-\beta/\e^{\beta'}).
\ee
Of course $\|\tilde g_\e\|_{H^k} = |\tilde g^\e|$ if $\tilde g^\e$ is a constant.
\end{definition}
\begin{lemma} \label{lem3}
 For $j,k,m \geq 0$, we find that
 \be
      \big \|\partial^j_t \partial^k_{\eta} \partial^m_{\xi} (\theta^0 - \bar \theta^0) \big \|_{L^2(D_{\f{1}{2}})} \text{ is an e.s.t.,}
 \label{eq_lem3}
 \ee
 for $t \in [0,T]$.
\end{lemma}
\begin{proof}
 For $m=0,1$, by \eqref{eq_est1}, we deduce that
 \be \label{lem3_eq2}
	  \begin{split}
	  \big \| \partial^k_{\eta} \partial^m_{\xi}  (\theta^0 - \bar \theta^0) \big \|^2_{L^2(D_{\f{1}{2}})} 
	    &  = \int^{2 \pi}_0 \int^{\f{1}{2}}_{\f{1}{4}} |\pa^k_{\eta} \pa^m_{\xi} [\theta^0(1-\delta)]|^2 d \xi d \eta \\
	    & \leq \int^{2 \pi}_0 \int^{\f{1}{2}}_{\f{1}{4}} \Big|\kappa \e^{-\f{m}{2}} \exp \Big(-\f{\xi^2}{4 \e t} \Big)\Big|^2 d \xi d \eta. \\
	  \end{split}
 \ee
 Noting that
 \be \notag
	  \int^{\f{1}{2}}_{\f{1}{4}} \Big|\exp \Big(- \f{\xi^2}{4 \e t} \Big) \Big|^2 d \xi
	  \leq  \kappa \bigg( \exp \Big(- \f{(1/4)^2}{4 \e t} \Big) \bigg)^2,
 \ee
 we obtain
 \be
	   \big \| \partial^k_{\eta} \partial^m_{\xi}  (\theta^0 - \bar \theta^0) \big \|_{L^2(D_{\f{1}{2}})} 
	   \text{ is an e.s.t..}
 \label{eq_est3}
 \ee
 For $j=0$, $m \geq 2$, or $j \geq 1$, $m \geq 0$,
 thanks to the pointwise estimate in \eqref{eq_est2},
 we find 
 \be
	  \big \| \partial^j_t \partial^k_{\eta} \partial^m_{\xi}  (\theta^0 - \bar \theta^0) \big \|^2_{L^2(D_{\f{1}{2}})}
	   \leq
	   \kappa \int^{2 \pi}_{0} \int_{\f{1}{4}}^{\f{1}{2}} \Big( \e^{-j-m+\f{1}{2}} \int^t_0 (1+s^{-2j-m+\f{1}{2}}) \exp \Big(-\f{\xi^2}{4 \e s} \Big)ds \Big)^2 d \xi d \eta.
  \label{eq_est_bdd}	   
 \ee
 We then focus on the problematic part in \eqref{eq_est_bdd}: 
 \be
	  \int^{2 \pi}_{0} \int_{\f{1}{4}}^{\f{1}{2}} \Big ( \int^t_0 s^{-l} \exp \Big( -\f{\xi^2}{4 \e s} \Big)  ds \Big)^2 d \xi d \eta, 
 \ee
 where $l=2j+m-\f{1}{2}$. Since $x^r \exp(-x^2) \leq \kappa \exp(-x^2/2)$ for $r>0$, where $\kappa$ depends only on $r$, we find
 \be
 \begin{split}
	  & \int^t_0 s^{-l} \exp \Big( -\f{\xi^2}{4 \e s} \Big) ds \\
	  & = \int^t_0 s^{-l} \exp \Big( -\f{\xi^2}{4 \e s} \Big) \Big(\f{\xi}{\sqrt{4 \e s}} \Big)^{2l} \Big(\f{\xi}{\sqrt{4 \e s}} \Big)^{-2l} ds \\
	  & \leq \kappa \int^t_0 s^{-l} \exp \Big( -\f{\xi^2}{ 8 \e s} \Big) \Big(\f{\xi}{\sqrt{4 \e s}} \Big)^{-2l} ds.
 \end{split}	  
 \label{eq2.21}
 \ee
 Hence, we obtain
 \be
 \begin{split}
	 & \int^{2 \pi}_{0} \int^{\f{1}{2}}_{\f{1}{4}} \Big( \int^t_0 s^{-l} \exp \Big( -\f{\xi^2}{4 \e s} \Big) ds \Big)^2 d \xi d \eta \\
	 & \leq \text{ (by }\eqref{eq2.21})\\
	 & \leq \kappa \int^{2 \pi}_{0} \int^{\f{1}{2}}_{\f{1}{4}}  \Big ( \int^t_0 s^{-l} \exp \Big( -\f{\xi^2}{ 8 \e s} \Big) \Big(\f{\xi}{\sqrt{4 \e s}} \Big)^{-2l} ds \Big)^2 d \xi d \eta \\
	 & \leq \kappa \int^{2 \pi}_{0} \int^{\f{1}{2}}_{\f{1}{4}} \exp \Big( -\f{\xi^2}{ 4 \e t} \Big) \Big(\f{\xi}{\sqrt{4 \e}} \Big)^{-4l} d \xi d \eta \\
	 & \leq \kappa \exp \Big( - \f{(1/4)^2}{4 \e t} \Big)\\
	 & \leq e.s.t. \text{  (for $t \in [0,T]$)}.
 \end{split}	 
 \ee
Then, the lemma follows.
\end{proof}
For the error analysis we borrow the following lemma from \cite{JPT11}.
\begin{lemma}
 Assume that the compatibility conditions \eqref{eq:comp} and \eqref{e:comp_cond2} hold. For $0 \leq m \leq 4$ and $k \geq 0$, there exists a positive constant $\kappa$ independent of $\epsilon$ such that
 \be
	  \Big \| \f{\partial^{m+k} \theta^0}{\partial \xi^m \partial \eta^k}\Big\|_{L^2(D_{\f{1}{2}}) }
	  \leq \kappa \epsilon^{-\f{m}{2}+\f{1}{4}},
	  \quad \text{for a.e.  } t \in [0,T].
 \ee
 \label{lem_pwbd3}
\end{lemma}
We now define ``the error'' $w^0_\e = w^0_\e (\eta,\xi,t)= u^\e - u^0 - \bar \theta^0$; then from \eqref{eq_main}, \eqref{eq_lim}, we deduce that
\be
\begin{cases}
\dfrac{\pa w^0_\e}{\pa t} - \e \Delta w^0_\e = \e \Delta u^0 - L_\e(\bar \theta^0), \\
w^0_\e(\eta, \xi=0,t) = 0,\\
w^0_\e (\eta,\xi,t=0) = 0,
\end{cases} \label{eq_w0}
\ee
where $L_\e$ is as in \eqref{eq_main_chg}.
We multiply $\eqref{eq_w0}_1$ by $w^0_\e$ and integrate over $D$; then we obtain
\be
\begin{split}
	  & \f{1}{2} \f{d}{dt} \| w^0_\e \|^2_{L^2(D)} + \e \| \nabla w^0_\e \|_{L^2(D)} \\
	  &  \leq \e \| w^0_\e \|_{L^2(D)} \|\Delta u^0\|_{L^2(D)} + \| w^0_\e \|_{L^2(D)} \|L_\e (\bar \theta^0 ) \|_{L^2(D)}\\
	  & \leq \| w^0_\e \|^2_{L^2(D)} + \f{\e^2}{2} \|\Delta u^0\|^2_{L^2(D)} + \f{1}{2}\|L_\e (\bar \theta^0 ) \|^2_{L^2(D)}.
\end{split}	  
\label{eq_a_priori1}	  
\ee
We note that
\be
\begin{split}
	  \| L_\e (\bar \theta^0 ) \|^2_{L^2(D)}  & = \|L_\e (\bar \theta^0 ) \|^2_{L^2(D^*)}  = \|  L_\e (\bar \theta^0 ) \|^2_{L^2(D_{\f{1}{2}})}  \\
	  & \leq  \|  L_\e (\bar \theta^0 - \theta^0 ) \|^2_{L^2(D_{\f{1}{2}})} + \|  L_\e (\theta^0 ) \|^2_{L^2(D_{\f{1}{2}})}  \\
	  & \leq e.s.t. + \|  L_\e (\theta^0 ) \|^2_{L^2(D_{\f{1}{2}})},
\end{split}
\label{eq_split}
\ee
and we rewrite $L_\e(\theta^0)$, using $\eqref{eq_cor}_1$, as
\be \notag
	  L_\e(\theta^0) = - \dfrac{\e}{(1-\xi)^2}\frac{\partial^2 \theta^0}{\pa \eta^2} + \dfrac{\e}{1-\xi}\dfrac{\partial \theta^0}{\pa \xi}.
\ee
Then, from Lemma \ref{lem_pwbd3}, we find
\be
	  \|L_\e(\theta^0)\|_{L^2(D_{\f{1}{2}})} \leq \kappa \e^{\f{5}{4}} + \kappa \e^{\f{3}{4}} \leq \kappa \e^{\f{3}{4}}.
\label{eq_L_bound}	 
\ee
From \eqref{eq_split} and \eqref{eq_L_bound}, and using the regularity $|\Delta u^0|^2_{L^2(D)} \leq \kappa$, \eqref{eq_a_priori1} becomes
\be
	  \f{1}{2} \f{d}{dt} \| w^0_\e \|^2_{L^2(D)} + \e \| \nabla w^0_\e \|_{L^2(D)} \leq \| w^0_\e \|^2_{L^2(D)} + \kappa \e^{\f{3}{2}}.
\label{eq_a_priori2}	  
\ee
Using the Gronwall inequality, we find
\be
	  \|w^0_\e\|_{L^{\infty}(0,T;L^2(D))} \leq \kappa \e^{\f{3}{4}}.
\label{eq_est_L2}
\ee
Integrating \eqref{eq_a_priori2} over $[0,T]$, we obtain
\be
	  \int^T_0 \| \nabla w^0_\e \|^2_{L^2(D)} dt \leq \f{1}{\e} \int^T_0 (\|w^0_\e\|^2_{L^2(D)} + \kappa \e^{\f{3}{2}}) dt.
\ee
Hence, we find
\be
	  \|w^0_\e\|_{L^{2}(0,T;H^1(D))} \leq \kappa \e^{\f{1}{4}}.
\label{eq_est_H1_L2}
\ee
\begin{remark}
  In \cite{GHT10}, one can also find the same convergence results for $\|w^0_\e\|_{L^{\infty}(0,T;L^2(\Omega))}$ and $\|w^0_\e\|_{L^{2}(0,T;H^1(\Omega))}$ where $\Omega$ is a general smooth domain. Hence, \eqref{eq_est_L2} and \eqref{eq_est_H1_L2} are special cases of the results in \cite{GHT10}. However to develop the convergence analysis for the finite elements space, we need an estimate on $\|w^0_\e\|_{L^{\infty}(0,T;H^1(D))}$ and on $\|w^0_\e\|_{L^{\infty}(0,T;H^2(D))}$ which do not appear in \cite{GHT10}; see below.
  \label{rem1}
\end{remark}
We now look for estimates of $\|w^0_\e\|_{L^{\infty}(0,T;H^1(D))}$ and $\|w^0_\e\|_{L^{\infty}(0,T;H^2(D))}$ which play an important role in Section \ref{sec_fem}. We take the scalar product of $\eqref{eq_w0}_1$ in the space $L^2(D)$ with $-\Delta w^0_\e$, and we obtain
\be
\begin{split}
	    & \frac{1}{2} \frac{d}{dt} \|\nabla w^0_\e\|^2_{L^2(D)} + \e \|\Delta w^0_\e\|^2_{L^2(D)} \\
	    & \leq \e \|\Delta u^0\|_{L^2(D)} \| \Delta w^0_\e \|_{L^2(D)} + \|L_\e(\bar \theta^0 )\|_{L^2(D)} \|\Delta w^0_\e \|_{L^2(D)}\\
	    & \leq \frac{\e}{4}\|\Delta w^0_\e\|^2_{L^2(D)} + \e \|\Delta u^0\|^2_{L^2(D)} + \frac{\e}{4}\|\Delta w^0_\e\|^2_{L^2(D)} + \f{1}{\e}\|L_\e (\bar \theta^0) \|^2_{L^2(D)}.	    
\end{split}
\label{eq_est4}
\ee
Using \eqref{eq_split}, \eqref{eq_L_bound} and \eqref{eq_est4}, we find
\be
	   \frac{1}{2} \frac{d}{dt} \|\nabla w^0_\e\|^2_{L^2(D)} + \e \|\Delta w^0_\e\|^2_{L^2(D)} 
	   \leq \frac{\e}{2}\|\Delta w^0_\e\|^2_{L^2(D)} + \kappa \e^{\f{1}{2}}.
\label{eq_est5}	
\ee
Hence, by the Gronwall inequality, we obtain
\be
	  \|w^0_\e\|_{L^{\infty}(0,T;H^1(D))} \leq \kappa \e^{\f{1}{4}}.
\ee
Furthermore, integrating \eqref{eq_est5} over $[0,T]$, we also find
\be
	  \|w^0_\e\|_{L^{2}(0,T;H^2(D))} \leq \kappa \e^{-\f{1}{4}}.
\ee
To find the estimate on $\|w^0_\e\|_{L^{\infty}(0,T;H^2(D))}$, we first take the time derivative of \eqref{eq_w0} and write
\be
\begin{cases}
	    \dfrac{\partial^2 w^0_\e}{\partial t^2} - \e \Delta \big( \dfrac{\partial w^0_\e}{\partial t} \big)
	    = \e \Delta \big( \dfrac{\partial u^0}{\partial t} \big) - L_\e \big(\dfrac{\partial \bar \theta^0}{\partial t} \big),\\
	     \dfrac{\partial w^0_\e}{\partial t}(\eta, \xi=0,t) = 0,\\
	      \dfrac{\partial w^0_\e}{\partial t}(\eta,\xi,t=0) = \e \Delta u^0.
\end{cases}
\label{e:2nd_t_deri}
\ee
\begin{remark}
We derive the initial condition $\eqref{e:2nd_t_deri}_3$ using \eqref{eq_cor} and \eqref{eq_w0}.
We consider $\eqref{eq_w0}_1$ at $t=0$, we then obtain
\be
	  \f{\partial w^0_\e}{\partial t} \Big |_{t=0}
	  = \Big (
		    \e \Delta w^0_\e + \e \Delta u^0 - L_\e(\bar \theta^0)
	    \Big )
	    \Big |_{t=0}.
\label{e:w0_t0}	    
\ee
According to $\eqref{eq_w0}_3$, $w^0_\e$ vanishes identically at $t=0$ so that
\be
	  \Delta w^0_\e = 0, \text{  at  } t=0.
\ee
Using \eqref{eq_main_chg} the term $L_\e(\bar \theta^0)$ in \eqref{e:w0_t0} becomes:
\be
	  L_\e (\bar \theta^0)
	  = 
	  \f{\partial \bar \theta^0}{\partial t}
	  - \f{\e}{(1-\xi)^2}\f{\partial^2 \bar \theta^0}{\partial \eta^2}\
	  + \f{\e}{1-\xi} \f{\partial \bar \theta^0}{\partial \xi}
	  - \e \f{\partial^2 \bar \theta^0}{\partial \xi^2}.
\label{e:L_t0_1}	  
\ee
According to $\eqref{eq_cor}_3$ and \eqref{e:cut_cor} the spatial derivatives of $\theta^0$ at $t=0$ vanish, that is,
\be
	  \dfrac{\partial^2 \bar \theta^0}{\partial \eta^2}  =0, \quad  
	  \dfrac{\partial \bar \theta^0}{\partial \xi}  =0, \quad \text{and} \quad  
	  \dfrac{\partial^2 \bar \theta^0}{\partial \xi^2}  = 0,   \text{  at  } t=0.
\label{e:L_t0_2}	  	  
\ee
We deduce, from $\eqref{eq_cor}_1$ and \eqref{e:cut_cor}, that
\be
	  \f{\partial \bar \theta^0}{\partial t} 
	  = \delta(\xi) \f{\partial \theta^0}{\partial t}
	  = \delta(\xi) \e \f{\partial^2 \theta^0}{\partial \xi^2}
	  = 0, \text{  at  } t= 0.\
\label{e:L_t0_3}	  	  
\ee
From \eqref{e:L_t0_1}, \eqref{e:L_t0_2}, and \eqref{e:L_t0_3} we obtain
\be
	  L_\e(\bar \theta^0) = 0, \text{  at  } t=0.
\ee
Hence we arrive at:
\be
	  \f{\partial w^0_\e}{\partial t} = \e \Delta u^0, \text{  at  } t=0.
\ee
\end{remark}
We take the scalar product of $\eqref{e:2nd_t_deri}_1$ in the space $L^2(D)$ with $\f{\partial w^0_\e}{\partial t}$:
\be
\begin{split}
	  & \f{1}{2} \f{d}{dt} \Big\| \f{\partial w^0_\e}{\partial t} \Big \|^2_{L^2(D)} 
	  + \e \Big \| \nabla \Big ( \f{\partial w^0_\e}{\partial t} \Big ) \Big \|^2_{L^2(D)}  \\
	  & = 
	  \e \Big ( \Delta \Big ( \f{\partial u^0}{\partial t} \Big ),\f{\partial w^0_\e}{\partial t} \Big )
	  - \Big (L_\e \Big ( \f{\partial \bar \theta^0}{\partial t} \Big ),\f{\partial w^0_\e}{\partial t} \Big ) \\
	  & \leq
	  \kappa \e^2 
	  + \f{1}{4} \Big \| \f{\partial w^0_\e}{\partial t} \Big \|^2_{L^2(D)}
	  + \Big \| L_\e \Big( \f{\partial \bar \theta^0}{\partial t} \Big) \Big \|^2_{L^2(D)}
	  + \f{1}{4} \Big \| \f{\partial w^0_\e}{\partial t} \Big \|^2_{L^2(D)}.
\end{split}
\label{e:2nd_last}
\ee
Using the similar argument as in \eqref{eq_split}
\be
\begin{split}
	  \Big \| L_\e \Big( \f{\partial \bar \theta^0}{\partial t} \Big ) \Big \|^2_{L^2(D)}
	  & \leq 
	  \Big \| L_\e \Big( \f{\partial \bar \theta^0}{\partial t} - \f{\partial \theta^0}{\partial t} \Big ) \Big \|^2_{L^2(D_{\f{1}{2}})}
	  + \Big \| L_\e \Big( \f{\partial \theta^0}{\partial t} \Big ) \Big \|^2_{L^2(D_{\f{1}{2}})} \\
	  & \leq \text{(by Lemma \ref{lem3})}\\
	  & \leq e.s.t. + + \Big \| L_\e \Big( \f{\partial \theta^0}{\partial t} \Big ) \Big \|^2_{L^2(D_{\f{1}{2}})}.
\end{split}
\ee
Noting that
\be
\begin{split}
	  \Big \| L_\e \Big( \f{\partial \theta^0}{\partial t} \Big ) \Big \|^2_{L^2(D_{\f{1}{2}})}
	  & \leq
	  \kappa \Big \| \e \f{\partial^3 \theta^0}{\partial t \eta^2} \Big \|^2_{L^2(D_{\f{1}{2}})}
	  + \kappa \Big \| \e \f{\partial^2 \theta^0}{\partial t \xi} \Big \|^2_{L^2(D_{\f{1}{2}})} \\
	  & \leq
	  \kappa \Big \| \e^2 \f{\partial^4 \theta^0}{\partial \xi^2 \eta^2} \Big \|^2_{L^2(D_{\f{1}{2}})}
	  + \kappa \Big \| \e^2 \f{\partial^3 \theta^0}{\partial \xi^3} \Big \|^2_{L^2(D_{\f{1}{2}})} \\	  
	  & \leq \text{(by Lemma \ref{lem_pwbd3})} \\
	  & \leq \kappa \e^{\f{3}{2}},
\end{split}
\ee
then \eqref{e:2nd_last} becomes
\be
	  \f{1}{2} \f{d}{dt} \Big\| \f{\partial w^0_\e}{\partial t} \Big \|^2_{L^2(D)} 
	  + \e \Big \| \nabla \Big ( \f{\partial w^0_\e}{\partial t} \Big ) \Big \|^2_{L^2(D)} 
	  \leq \kappa \e^{\f{3}{2}} + \f{1}{2} \Big\| \f{\partial w^0_\e}{\partial t} \Big \|^2_{L^2(D)}.
\ee
Thanks to the Gronwall inequality, we obtain
\be
	  \Big\| \f{\partial w^0_\e}{\partial t} \Big \|_{L^2(D)} \leq \kappa \e^{\f{3}{4}}.
\label{e:w_deri_est}	  
\ee
From $\eqref{eq_w0}_1$, we deduce that
\be
\begin{split}
	  \|w^0_\e \|_{L^{\infty}(0,T;H^2(D))}
	  & \leq
	  \f{1}{\e} \Big \|\f{\partial w^0_\e}{\partial t} \Big \|_{L^{\infty}(0,T;L^2(D))}
	  + \| \Delta u^0\|_{L^{\infty}(0,T;L^2(D))}
	  + \f{1}{\e}  \| L_\e (\bar \theta^0 )   \|_{L^{\infty}(0,T;L^2(D))} \\
	  & \leq \text{(by \eqref{e:w_deri_est})} \\
	  & \leq  
	  \kappa \e^{-\f{1}{4}} 
	  + \kappa 
	  + \f{\kappa}{\e}  \| L_\e  (\bar \theta^0 - \theta^0 )   \|_{L^{\infty}(0,T;L^2(D))}
	  + \f{\kappa}{\e}  \| L_\e  (\theta^0 )   \|_{L^{\infty}(0,T;L^2(D))} \\
	  & \leq \text{(by \eqref{eq_L_bound})} \\
	  & \leq \kappa \e^{-\f{1}{4}}.
\end{split}
\ee
Hence, we finally obtain
\be
	  \| w^0_\e \|_{L^{\infty}(0,T;H^2(D))} \leq \kappa \e^{-\f{1}{4}}.
\ee
We then arrive at the following conclusion.
\begin{theorem}
 Let $u^\e$ be the solution of \eqref{eq_main} and $u^0$ be the solution of \eqref{eq_lim}. Then, the following estimates hold:
 \be
 \begin{split}
	  & \|u^\e - u^0 - \bar \theta^0 \|_{L^{\infty}(0,T;L^2(D))} \leq \kappa \e^{\f{3}{4}},\\
	  & \|u^\e - u^0 - \bar \theta^0 \|_{L^{\infty}(0,T;H^1(D))} \leq \kappa \e^{\f{1}{4}},\\
	  & \|u^\e - u^0 - \bar \theta^0 \|_{L^{\infty}(0,T;H^2(D))} \leq \kappa \e^{-\f{1}{4}},
 \end{split}
 \label{eq_thm1}
 \ee
 where $\bar \theta^0$ is the corrector in \eqref{e:cut_cor} and $\kappa$ is a constant independent of $\epsilon$.
 \label{main_thm1}
\end{theorem}

\section{\bf Approximation via finite elements} \label{sec_fem}
In this section, we show how to approximate \eqref{eq_main} using P1 finite elements with or without an enriched space.
We first introduce the classical finite elements spaces $V_N$ and then define the new finite element spaces $\big ( \bar V_{Nt} \big )_{t \in [0,T]}$ enriched with the boundary layer elements and the corresponding approximations. Then, we develop the convergence analysis for the new scheme.

\subsection{Finite element spaces}
We define the standard finite element spaces $V_N$ such that
\be
	  V_N := \Big \{ \sum^N_{i=1} c_i \varphi_i(x,y)  \Big \} \subset H^1_0(\Omega),
\ee
where the $\varphi_i(x,y)$ is the classical P1 functions equal to $1$ at some node $M_i$ and to $0$ at all other nodes with $1 \leq i \leq N$.
We then introduce the new finite element space supplemented with the boundary layer elements:
\be
	\big ( \bar V_{Nt} \big )_{t \in [0,T]} := \Big \{ \sum^N_{i=1} c_i(t) \varphi_i(x,y) + \sum^M_{j=1} d_j(t) \varphi_0(\xi,t) \psi_j(\eta)     \Big \},
\ee
where the $\psi_j(\eta)$ are the classical P1 elements in 1D space, i.e. the hat functions, for $1 \leq j \leq M$. Here $\varphi_0 = \varphi_0(\xi,t)$ is the boundary layer element, that is,
\be
	    \varphi_0(\xi,t) = \Big (1- \int^t_0   I(\xi,t-s) ds \Big ) \delta(\xi),
\label{bl_element}	    
\ee
where $I=I(\xi,t)$ is the function in $\eqref{eq_err_fuc}_1$ and $\delta = \delta(\xi)$ is a smooth cut-off function as in Section \ref{sec2}.\\
We aim to study the classical and new approximation solutions $u^{\e}_N \in V_N$ and $\bar u^{\e}_N \in \big ( \bar V_{Nt} \big )_{t \in [0,T]}$, respectively, such that
\be
\begin{split}
	  & u^\e_N: [0,T] \longrightarrow V_N,\\
	  & (\pa_t u^{\e}_N,v) +   (\nabla u^{\e}_N,\nabla v) = (f,v), \quad \forall v \in V_N, \text{  } t \in (0,T],\\
	  & (u^\e_N(0),v) = (u_0,v), \quad \forall v \in V_N,
\end{split}
\label{eq_var_form1}	  
\ee
and
\be
\begin{split}
	  & \bar u^\e_N: [0,T] \longrightarrow \big ( \bar V_{Nt} \big )_{t \in [0,T]},\\
	  & (\pa_t \bar u^{\e}_N,v) +   (\nabla \bar u^{\e}_N, \nabla v) = (f,v), \quad \forall v \in \big ( \bar V_{Nt} \big )_{t \in [0,T]}, \text{  } t \in (0,T],\\
	  & (\bar u^\e_N(0),v) = (u_0,v), \quad \forall v \in \big ( \bar V_{Nt} \big )_{t=0}.
\end{split}
\label{eq_var_form2}
\ee

\subsection{Convergence analysis} \label{sec3.2}
In this section, we study the convergence analysis for the finite element approximation using the results in Section \ref{sec2}. Thanks to $\eqref{eq_thm1}_5$ and the regularity assumption 
\be
\| u^0 \|_{L^{\infty}(0,T;H^2(D))} \leq \kappa,
\ee
where $u^0 = u^0(\eta,0,t)$, we obtain 
\be
	    \| u^\e - \bar  \theta^0 \|_{L^{\infty}(0,T;H^2(D))} \leq \kappa \e^{-\f{1}{4}}.
\label{eq_conv_case1}	    
\ee
Setting $g(\eta,t) = u^0(\eta,0,t)$ and $I = I(\xi,t-s)$, we then have:
\be
\begin{split}
	  & \| u^\e - g \varphi_0 \|_{L^{\infty}(0,T;H^2(D))} \\
	  & = \| u^\e - \bar \theta^0 + \bar \theta^0 - g \varphi_0 \|_{L^{\infty}(0,T;H^2(D))} \\
	  & \leq \| u^\e - \bar \theta^0 \|_{L^{\infty}(0,T;H^2(D))} 
		    + \| \bar \theta^0 - g \varphi_0 \|_{L^{\infty}(0,T;H^2(D))}\\
	  & \leq \kappa \e^{-\f{1}{4}} 
		+ \Big \| - \Big (\int^t_0 I \f{\pa u^0}{\pa t}(\eta,0,s) ds \Big ) \delta
			      -  g \Big (1 - \int^t_0 I ds \Big ) \delta 
			      \Big \|_{L^{\infty}(0,T;H^2(D_{\f{1}{2}}))} \\
	  & \leq \kappa \e^{-\f{1}{4}} 
		 + \Big \| \int^t_0 I \Big ( g - \f{\pa u^0}{\pa t}(\eta,0,s)\Big ) ds 
			  \Big \|_{L^{\infty}(0,T;H^2(D_{\f{1}{2}}))} + \kappa \| g \delta \|_{L^{\infty}(0,T;H^2(D_{\f{1}{2}}))}\\
	  & \leq \kappa \e^{-\f{1}{4}} + \mathcal{R},
\end{split}
\label{eq_conv_case2}
\ee
where 
$$\mathcal{R} = \Big \| \int^t_0 I \Big (g - \f{\pa u^0}{\pa t}(\eta,0,s) \Big ) ds  \Big \|_{L^{\infty}(0,T;H^2(D_{\f{1}{2}}))}.$$
To estimate the term $\mathcal{R}$, we consider only the dominating term which is the second derivative in $\xi$, i.e. $\f{\pa^2}{\pa \xi^2}$. We note that
\be
	  \int^t_0 I \Big(g - \f{\partial u^0}{\partial t}(\eta,0,s) \Big) ds
	  = g \int^t_0 I ds + \theta^0;
\ee
hence we deduce that
\be
\begin{split}
	  \mathcal{R} 
	  & \leq \| \theta^0 \|_{L^{\infty}(0,T;H^2_{\xi}(D_{\f{1}{2}})) }
	  + \Big \| g \int^t_0 I ds \Big \|_{L^{\infty}(0,T;H^2_{\xi}(D_{\f{1}{2}}))} \\
	  & \leq \kappa \e^{-\f{3}{4}} + \kappa \Big \|\int^t_0 I ds \Big \|_{L^{\infty}(0,T;H^2_{\xi}(D_{\f{1}{2}}))}.
\end{split}	  
\ee
We apply Lemma 2.1 in \cite{JPT11}, which is the generalized version of Lemman \ref{lem_pwbd3} in this article, we then obtain
\be
	 \Big \|\int^t_0 I ds \Big \|_{L^{\infty}(0,T;H^2_{\xi}(D_{\f{1}{2}}))} \leq \kappa \e^{-\f{3}{4}}.
\ee
Hence, we obtain
\be
	  \mathcal{R} \leq \kappa \e^{-\f{3}{4}}.
\label{eq_conv_case3}	  
\ee
Hence, from \eqref{eq_conv_case1}, \eqref{eq_conv_case2}, and \eqref{eq_conv_case3}, we find
\be
	    \| u^\e - g \varphi_0 \|_{L^{\infty}(0,T;H^2(D))} \leq \kappa \e^{-\f{3}{4}}.
\label{eq_bdd_0}	    
\ee
For further analysis, we now prove the following interpolation lemmas.
\begin{lemma} \label{lem3.0}
\leavevmode
  Let $h$ be the one-dimensional mesh size.
\begin{enumerate}
  \item 
  Assume that $\gamma \in H^l(0,2 \pi)$ for $l=1,2$. Then, there exist $c_i \in \mathbb{R}$, $i=1,...,N_1$, such that for $m=0$ if $l=1$ and $m=0,1$ if $l=2$
  \be
	      \Big \|\gamma - \sum^{N_1}_{i=1} c_i \psi_i \Big \|_{H^m(0,2 \pi)} \leq \kappa h^{l-m}\| \gamma \|_{H^l(0,2\pi)}.
  \ee
 
  \item 
  Assume that $\gamma \in C([0,T];H^l)$. Then, there exist $c_i=c_i(t) \in C([0,T])$, $i=1,...,N_1$, such that for $m=0$ if $l=1$ and $m=0,1$ if $l=2$
  \be
	      \Big \|\gamma(t) - \sum^{N_1}_{i=1} c_i(t) \psi_i \Big \|_{H^m(0,2 \pi)} \leq \kappa h^{l-m} \sup_{t \in [0,T]} \| \gamma(t) \|_{H^l(0,2\pi)}.
  \ee
\end{enumerate}
\end{lemma}
\begin{proof}
\leavevmode
 \begin{enumerate}
  \item 
  The result is classical and the proof can be found, e.g., in \cite{CP02}.
  \item 
  Now the $c_i$ depend on $t$, $c_i=c_i(t)$ and we need to analyze the dependence in $t$ of the $c_i$.
  A perusal of the proof in \cite{CP02} shows that the result hinges on the regularity in time of the interpolation mappings $\Pi_K$:
  \be
	    \gamma 
	    = \gamma(\cdot,t) \longrightarrow (\Pi_K \gamma) \Big ( \f{i}{N_1},t \Big ) 
	    = \gamma \Big (\f{i}{N_1},t \Big ), 
	    \quad i=0,...,N_1,
  \ee
  which map $H^1([0, 2 \pi])$ into $C([0, 2 \pi])$ and $H^{r+1}([0, 2 \pi])$ into $C^r([0, 2 \pi])$ for $r \geq 0$.
  Hence the continuity in time of the interpolants in $H^r$ follows when $\gamma \in C([0,T];H^r([0, 2 \pi]))$.
 \end{enumerate}
\end{proof}

\begin{lemma} \label{lem3.1}
\leavevmode
Let $h$ be the two-dimensional mesh size, that is the maximum diameter of the triangular elements.
 \begin{enumerate}
  \item  
	  Assume that $w \in H^2(D)$. Then, there exist $c_j \in \mathbb{R}$, $j=1,...,N_2$, such that
	  \be
	      \Big \| w - \sum^{N_2}_{j=1} c_j \varphi_j \Big \|_{H^m(D)} 
	      \leq \kappa h^{2-m} \big \| w \big\|_{H^2(D)}.
	  \ee
  \item
	  Assume that $w \in C([0,T];H^2(D))$. Then, there exist $c_j = c_j(t) \in C([0,T])$, $j=1,...,N_2$, such that
	  \be
	      \Big \| w(t) - \sum^{N_2}_{j=1} c_j(t) \varphi_j \Big \|_{H^m(D)} 
	      \leq \kappa h^{2-m} \sup_{t \in [0,T]} \big \| w(t) \big\|_{H^2(D)}.
	  \ee
 \end{enumerate}
\end{lemma}
\begin{proof}
 \leavevmode
 \begin{enumerate}
  \item 
  The result is classical and the proof can be found, e.g., in \cite{CP02}.
  \item 
  We use the same method as in Lemma \ref{lem3.0}. Considering the regularity of the interpolation mapping $\Pi_K$, we write
  \be
	    w = w(\cdot,t) \longrightarrow (\Pi_K w) (A_j,t ) = w(A_j,t),
  \ee
  which map $H^2(D)$ into $C(D)$, where the $A_j$ are the nodal points of the P1 elements.
  Hence the continuity in time of the interpolants in $H^2(D)$ follows when $w \in C([0,T];H^2(D))$.
 \end{enumerate}
\end{proof}

\begin{remark}
 The interpolation results in (1) of Lemma \ref{lem3.0} and \ref{lem3.1} are standard and the proofs have been presented in many other places; see e.g. \cite{BS94} and \cite{SF73}. However, in the case (2) of Lemmas \ref{lem3.0} and \ref{lem3.1}, we could not find the specific proofs in any earlier works although the results has been referred to in the literatures. Hence the proofs in Lemmas \ref{lem3.0} and \ref{lem3.1} are useful in this article and for the future works.
\end{remark}

\begin{lemma} \label{interpol_lem2}
 There exist $c_i=c_i(t) \in C([0,T])$ and $d_j=d_j(t) \in C([0,T])$, $i=1,...,N$, $j=1,...,M$, such that
 \be
 \begin{split}
	   & \Big \| u^\e - \sum^N_{i=1} c_i\varphi_i - \sum^M_{j=1}d_j \varphi_0 \psi_j	\Big \|_{L^{\infty}(0,T;L^2(D))} \leq \kappa h^2 \e^{-\f{3}{4}},\\
	   & \Big \| u^\e - \sum^N_{i=1} c_i\varphi_i - \sum^M_{j=1}d_j \varphi_0 \psi_j	\Big \|_{L^{\infty}(0,T;H^1(D))} \leq \kappa h \e^{-\f{3}{4}}.
 \end{split}
 \label{eq_H2_interpol}
 \ee
\end{lemma}
\begin{proof}
 Using \eqref{eq_bdd_0} and Lemma \ref{lem3.1}, for $m=0,1$, we deduce
 \be
 \begin{split}
	 \Big \|   u^\e - g \varphi_0 - \sum^N_{i=1} c_i \varphi_i	\Big \|_{L^{\infty}(0,T;H^m(D))} 
	 & \leq \kappa h^{2-m} \|u^\e - g \varphi_0\|_{L^{\infty}(0,T;H^2(D))} \\
	 & \leq \kappa h^{2-m} \e^{-\f{3}{4}}.
 \end{split}	 
 \ee
Moreover, by Lemmas \ref{lem_pwbd3} and \ref{lem3.0}, we also find
\be
  \begin{split}
	\Big \| g \varphi_0 - \sum^M_{j=1} d_j \varphi_0 \psi_j \Big \|_{L^{\infty}(0,T;H^m(D))} 
	& = \Big \| \varphi_0 (g - \sum^M_{j=1} d_j \psi_j )\Big \|_{L^{\infty}(0,T;H^m(D_{\f{1}{2}}))}\\
	& \leq \kappa h^{2-m}  \| \varphi_0  \|_{L^{\infty}(0,T;H^2_{\xi}(0,\f{1}{2}))}\\
	& \leq \kappa h^{2-m} \e^{-\f{3}{4}}.
  \end{split}
\ee
Then, Lemma \ref{interpol_lem2} follows.
\end{proof}
We now prove the main convergence theorem of Section \ref{sec_fem} concerning the new scheme. See Remark \ref{rem:3_2} below comparing to the result for the classical scheme.
\begin{theorem} \label{thm:main}
Let $u^\e$ and $\bar u_N^\e$ be the solutions of \eqref{eq_main_var_form} and \eqref{eq_var_form2}, respectively, then 
 \be
	    \| u^\e - \bar u^\e_N \|_{L^{\infty}(0,T;L^2(D))} \leq \kappa h^2 \e^{-\f{3}{4}} \Big( 1+ \log \f{T}{h^2} \Big ) + \kappa h \e^{-\f{1}{4}}.
 \label{eq_main_result}	    
 \ee
\end{theorem}
\begin{proof}
For the proof, we apply the Aubin-Nitsche trick as in \cite{AJ67}, \cite{CTW00}, \cite{CJ09} and \cite{NJ70}.
We first set
\be
\hat u^\e_N = \sum^N_{i=1} c_i \varphi_i + \sum^M_{j=1} d_j \varphi_0 \psi_j.
\ee
Then, by Lemma \ref{interpol_lem2}, we deduce that
\be
\begin{split}
	    \| u^\e - \bar u^\e_N \|_{L^{\infty}(0,T;L^2(D))} 
	    & \leq \| u^\e - \hat u^\e_N \|_{L^{\infty}(0,T;L^2(D))} + \| \hat u^\e_N - \bar u^\e_N \|_{L^{\infty}(0,T;L^2(D))} \\
	    & \leq \kappa \e^{-\f{3}{4}} h^2 + \| \hat u^\e_N - \bar u^\e_N \|_{L^{\infty}(0,T;L^2(D))}.
\end{split}
\label{eq_main_est1}
\ee
We now consider a duality argument to estimate $e^\e_N(t) := \hat u^\e_N - \bar u^\e_N$ in ${L^{\infty}(0,T;L^2(D))}$. For $t \in (0,T)$, let $\Phi^\e_N: (0,t) \longrightarrow \big ( \bar V_{Nt} \big )_{t \in [0,T]}$ satisfy
\be
\begin{cases}
	  -(\pa_s \Phi^\e_N(s),v) + a_\e(\Phi^\e_N(s),v) = 0, \quad 0<s<t, \quad v \in  \big ( \bar V_{Nt} \big )_{t \in [0,T]},\\
	  \Phi^\e_N(t) = e^\e_N(t),	  
\end{cases}
\label{eq_dual}
\ee
where $a_\e(u,v) = \e ( \nabla u, \nabla v)$. Taking $v = e^\e_N(s)$ in $\eqref{eq_dual}_1$, we find
\be
\begin{split}
	  \|e^\e_N(t)\|^2_{L^2(D)} 
	  & =  \int^t_0 \Big \{  -(\pa_s \Phi^\e_N(s),e^\e_N(s)) + a_\e (\Phi^\e_N(s),e^\e_N(s))  \Big \} ds + (\Phi^\e_N(t),e^\e_N(t))\\
	  & = \text{ (by integration by parts)}\\
	  & = \int^t_0 \Big \{  (\pa_s e^\e_N(s), \Phi^\e_N(s)) + a_\e (e^\e_N(s),\Phi^\e_N(s))  \Big \} ds + (\Phi^\e_N(0),e^\e_N(0)).
\end{split}
\label{eq_IBP_1}
\ee
From \eqref{eq_main_var_form} and \eqref{eq_var_form2}, we find
\be
\begin{cases}
	  (\pa_t (u^\e - \bar u^\e_N),v) + a_\e(u^\e - \bar u^\e_N,v)=0,\\
	  (u^\e(0) - \bar u^\e_N(0),v)=0.
\end{cases}
\label{eq_Gal_ortho}
\ee
Let us set $\Theta^\e_N(s) = u^\e(s) - \hat u^\e_N(s)$, then from \eqref{eq_IBP_1} and \eqref{eq_Gal_ortho} we obtain
\be
\begin{split}
	  \|e^\e_N(t)\|^2_{L^2(D)} 
	  & = \int^t_0 \Big \{ (\pa_s \Theta^\e_N,\Phi^\e_N) + a_\e (\Theta^\e_N,\Phi^\e_N) \Big \} ds + (\Theta^\e_N(0),\Phi^\e_N(0))\\
	  & = \text{ (by integration by parts)}\\
	  & = \int^t_0 \Big \{ - (\Theta^\e_N,\pa_s \Phi^\e_N) + a_\e (\Theta^\e_N,\Phi^\e_N) \Big \} ds + (\Theta^\e_N(t),\Phi^\e_N(t)).
\label{eq_essential}	  
\end{split}
\ee
To finish the proof, we look for estimates on the R.H.S. of \eqref{eq_essential}. Note that \eqref{eq_dual} is equivalent to the following ODE system
\be
\begin{cases}
	  - w' + \e A_h w = 0, \quad s \in (0,t)\\
	  w(t) = w_0,
\end{cases}
\label{eq_ODE_form}
\ee
where $w' = \f{d}{ds} w$ and $A_h$ is the discrete laplacian with respect to the spatial variables. We first take the scalar product of \eqref{eq_ODE_form} with $w$; we see that
\be
	  -\f{1}{2} \f{d}{ds} |w|^2 + \e |A_h^{\f{1}{2}}w|^2 =0.
\label{eq_ODE_form2}	  
\ee
Integrating \eqref{eq_ODE_form2} over $(s,t)$, we then find
\be
	  |w(s)| \leq |w(t)|, \quad \forall s \in (0,t);
\ee
hence we obtain
\be
	  \| \Phi^\e_N \|_{L^{\infty}(0,t;L^2(D))} \leq \kappa \| e^\e_N(t) \|_{L^2(D)}.
\label{eq_bdd_1}
\ee
Moreover, integrating \eqref{eq_ODE_form2} over $(0,t)$, we have
\be
	  \e \int^t_0 | A_h^{\f{1}{2}} w|^2 \leq |w(t)|^2.
\label{eq_bdd_1_2}	  
\ee
To find further estimates, it is convenient to consider the change of variable $\tau = t-s$. We then define $\tilde w(\tau) = w(t-s)$ and rewrite \eqref{eq_ODE_form} as
\be
\begin{cases}
	  \tilde w' + \e A_h \tilde w = 0,\\
	  \tilde w(\tau=0) = \tilde w_0 = w(t).
\label{eq_newODE_form}	  
\end{cases}
\ee
Multiplying \eqref{eq_newODE_form} by $\tau \tilde w'$, we obtain
\be
	  \tau | \tilde w' |^2 + \f{\e}{2}\f{d}{d \tau}(\tau |A_h^{\f{1}{2}} \tilde w |^2) - \f{\e}{2} |A_h^{\f{1}{2}} \tilde w|^2 = 0.
\label{eq_ODE2_est1}
\ee
Integrating \eqref{eq_ODE2_est1} over $(0,\tau)$ and using \eqref{eq_bdd_1_2}, we deduce
\be
	  \tau |A_h^{\f{1}{2}} \tilde w(\tau)|^2 \leq \int^{\tau}_0 |A_h^{\f{1}{2}} \tilde w|^2 \leq \int^t_0 |A_h^{\f{1}{2}} w|^2 \leq \f{\kappa}{\e} |w(t)|^2.
\label{eq_H1_L1_bdd}	  
\ee
Hence,
\be
	   |A_h^{\f{1}{2}}w(\tau)| \leq \f{\kappa}{\sqrt{ \e (t-s)}}  |w(t)|, \quad \forall \tau \in (0,t).
\label{eq_bdd_2}
\ee
Moreover, integrating \eqref{eq_ODE2_est1} over $(0,t)$ and using \eqref{eq_bdd_1_2}, we find
\be
	     \int^t_0 (\sqrt{\tau} |\tilde w'|)^2 ds + \f{\e t}{2} |A_h^{\f{1}{2}} \tilde w(t)|^2 \leq \f{\e}{2} \int^t_0 |A_h^{\f{1}{2}} \tilde w|^2 ds
\ee
implying that
\be
	      \|\sqrt{\tau} \tilde w'  \|^2_{L^2(0,t)} \leq \kappa |w(t)|^2.
\ee
Hence, we obtain
\be
	     \|\sqrt{\tau} \tilde w'  \|_{L^2(0,t)} \leq \kappa |w(t)|.
\label{eq_bdd_3}
\ee
We now consider the time derivative of \eqref{eq_newODE_form}
\be
	   \tilde w'' + \e A_h \tilde w' = 0,
\label{eq_newODE_deri}	  
\ee
and multiply \eqref{eq_newODE_deri} by $\tau^2 \tilde w'$
\be
	    \frac{1}{2} \f{d}{d \tau} |\tau \tilde w'|^2 - \tau |\tilde w'|^2 + \e|\tau A_h^{\f{1}{2}} \tilde w'|^2 = 0.
\label{eq_ODE3_est1}	    
\ee
Integrating \eqref{eq_ODE3_est1} over $(0,\tau)$ and using \eqref{eq_bdd_3}, we find
\be
	   \tau^2 |\tilde w'(\tau)|^2 \leq \kappa \int^t_0 \tilde \tau |\tilde w'|^2 d \tilde \tau \leq \kappa |w(t)|^2.
\ee
Hence,
\be
	  |\tilde w'(\tau)| \leq \kappa \tau^{-1} |w(t)|, \quad \forall \tau \in (0,t).
\label{eq_bdd_4}	  
\ee
Moreover, we also deduce
\be
\begin{split}
      \int^t_0 |\tilde w'| d\tau 
      & = \int^{h^2}_0 |\tilde w'| d \tau + \int^{t}_{h^2} |\tilde w'| d \tau \\
      & \leq \text{ (by \eqref{eq_newODE_form} and \eqref{eq_bdd_4})} \\
      & \leq \int^{h^2}_0 \e |A_h \tilde w| d \tau + \kappa \int^{t}_{h^2} \f{1}{\tau} |w(t)| d \tau \\
      & \leq \text{ (by the standard inverse Poincar\'e inequality  } |A_hw| \leq \kappa h^{-2}|w|)\\
      & \leq \int^{h^2}_0 \e h^{-2} | w(t)| d \tau + \kappa |w(t)| \log\f{T}{h^2}\\
      & \leq \kappa |w(t)| (\e + \log \f{T}{h^2}).
\end{split}      
\label{eq_bdd_5}
\ee
Back to \eqref{eq_essential}, we finally obtain
\be
\begin{split}
	  \|e^\e_N(t)\|^2_{L^2(D)} 
	  & \leq \int^t_0 - (\Theta^\e_N,\pa_s \Phi^\e_N) + a_\e (\Theta^\e_N,\Phi^\e_N) ds + (\Theta^\e_N(t),\Phi^\e_N(t)) \\
	  & \leq \| \Theta^\e_N \|_{L^{\infty}(0,t;L^2(D))} \|\pa_s \Phi^\e_N\|_{L^1(0,t;L^2(D))} 
	      + \e \|\Theta^\e_N\|_{L^{\infty}(0,t;H^1(D))} \|\Phi^\e_N\|_{L^1(0,t;H^1(D))}\\
	  &    \quad \quad + \| \Theta^\e_N(t) \|_{L^2(D)} \| e^\e_N(t) \|_{L^2(D)}\\
	  & \leq \text{ (by \eqref{eq_H2_interpol}, \eqref{eq_bdd_1}, \eqref{eq_H1_L1_bdd}, \eqref{eq_bdd_2}, and \eqref{eq_bdd_5})}\\
	  & \leq \kappa h^2 \e^{-\f{3}{4}} \Big (\e + \log \f{T}{h^2} \Big ) \| e^\e_N(t) \|_{L^2(D)} + \kappa h \e^{-\f{1}{4}}   \| e^\e_N(t) \|_{L^2(D)}\\
	  & \quad \quad + \kappa h^2 \e^{-\f{3}{4}} \| e^\e_N(t) \|_{L^2(D)}.
\end{split}
\label{eq_main_est2}
\ee
Hence, from \eqref{eq_main_est1} and \eqref{eq_main_est2}, we obtain
\be
 			 \| u^\e - \bar u^\e_N \|_{L^{\infty}(0,T;L^2(D))} \leq  \kappa h^2 \e^{-\f{3}{4}} \Big( 1+ \log 				\f{T}{h^2} \Big ) + \kappa h \e^{-\f{1}{4}}.
\ee
This completes the proof of Theorem \ref{thm:main}.
\end{proof}
\begin{remark}
 We recall the corrector $\bar \theta^0$:
 $$
	  \bar \theta^0 
	  = - \delta(\xi) \int^t_0 I(\xi,t-s) \f{\partial u}{\partial t}(\eta,0,s) ds,
 $$
which cannot be expressed by the separation of variables in space and time.
Hence we could not fully take advantage of $\eqref{eq_thm1}_5$ when we find the upper bound in \eqref{eq_bdd_0}.
Hence the upper bound of $\| u^\e - u^\e_N \|$, the convergence of the standard FEM, is the same as in Theorem \ref{thm:main}.
However, in view of numerical simulations, we easily see that the new scheme with the boundary layer elements is much more accurate than the standard one; see e.g. Figures \ref{fig_1d}-\ref{fig_2d_log}.
\label{rem:3_2}
\end{remark}

\section{\bf Numerical simulations} \label{sec_num}
In this section, we present the results of numerical simulations of \eqref{eq_main} using the standard finite element method (SFEM) and the new finite element method (NFEM), which correspond to \eqref{eq_var_form1} and \eqref{eq_var_form2}, respectively. 
\subsection{Modified boundary layer element}
In the simulations, we do not use the boundary layer element $\varphi_0$ in \eqref{bl_element} directly since the term $I(\xi,t-s)$ is not convenient for the integrations over a triangle.
Instead, we consider the modified boundary layer element $\tilde \varphi_0 = \tilde \varphi_0(\xi,t)$:
\be
	  \tilde \varphi_0 (\xi,t) = \Big [ 1 - \exp \Big (-\f{\xi^2}{4 \e t} \Big) \Big ] \delta(\xi).
 \label{app_bl_element1}
\ee
The approximation $\tilde \varphi_0$ is much easier to implement numerically in coding than the corrector $\varphi_0$.
More precisely, the approximation $\tilde \varphi_0$ produces less errors in the numerical integrations than the boundary layer element $\varphi_0$ of \eqref{bl_element};
One drawback of both boundary layer elements $\varphi_0$ and $\tilde \varphi_0$ is that we calculate the element matrices for each time step since the boundary layer element in \eqref{app_bl_element1} depends on time.
To improve computational efficiency, we introduce the time-independent boundary layer element $\tilde \varphi_{-1}$ such that
\be
	  \tilde \varphi_{-1} (\xi) = \Big [ 1 - \exp \Big (-\f{\xi^2}{4 \e} \Big) \Big ] \delta(\xi).
 \label{app_bl_element2}
\ee
As long as the final time is not very large, or $\e$ is small enough (for instance $T \e \sim 10^{-4}$), then the time-independent approximation of the boundary layer element is acceptable in the sense of numerical simulations as shown by the following results; see Figures \ref{comp_cor1} and \ref{comp_cor2}.
\begin{figure}[h]
    \centering
    \includegraphics[scale=0.47]{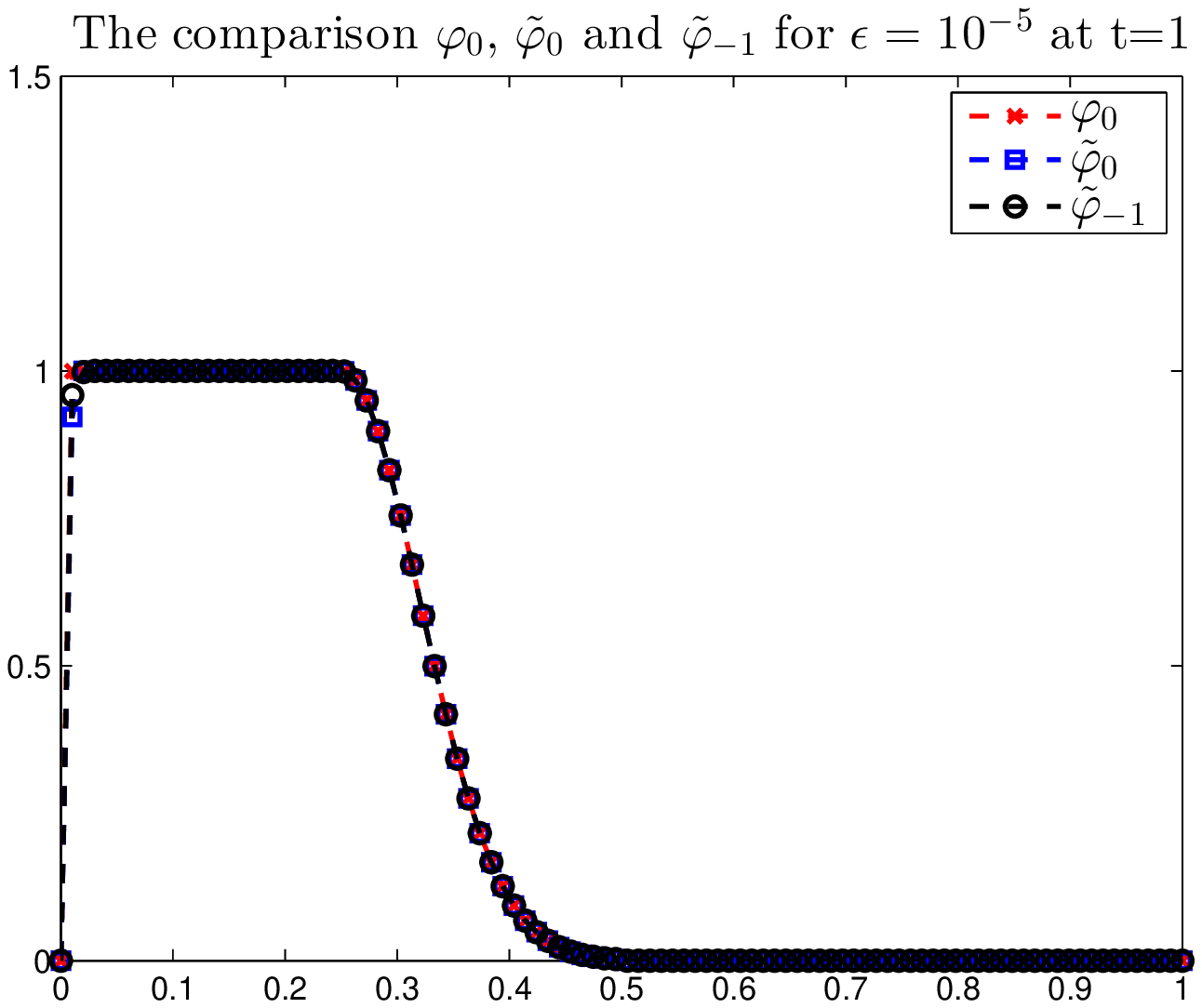}
    \includegraphics[scale=0.47]{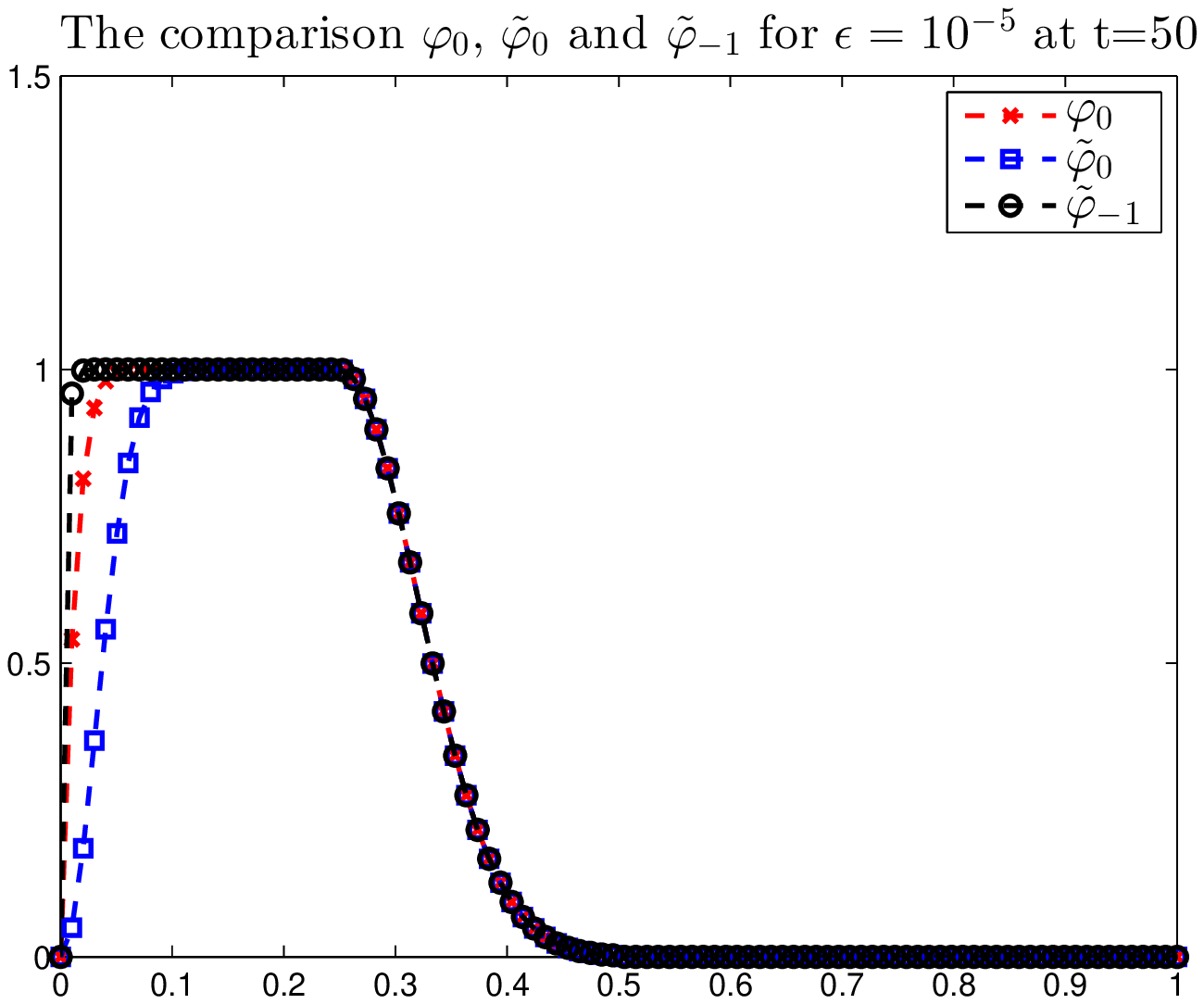}
    \caption{The comparison of $\varphi_0$, $\tilde \varphi_0$ and $\tilde \varphi_{-1}$ for $\e = 10^{-5}$ at $t=1$ and $t=50$.}
    \label{comp_cor1}
\end{figure}
\begin{figure}[h]
    \centering
    \includegraphics[scale=0.47]{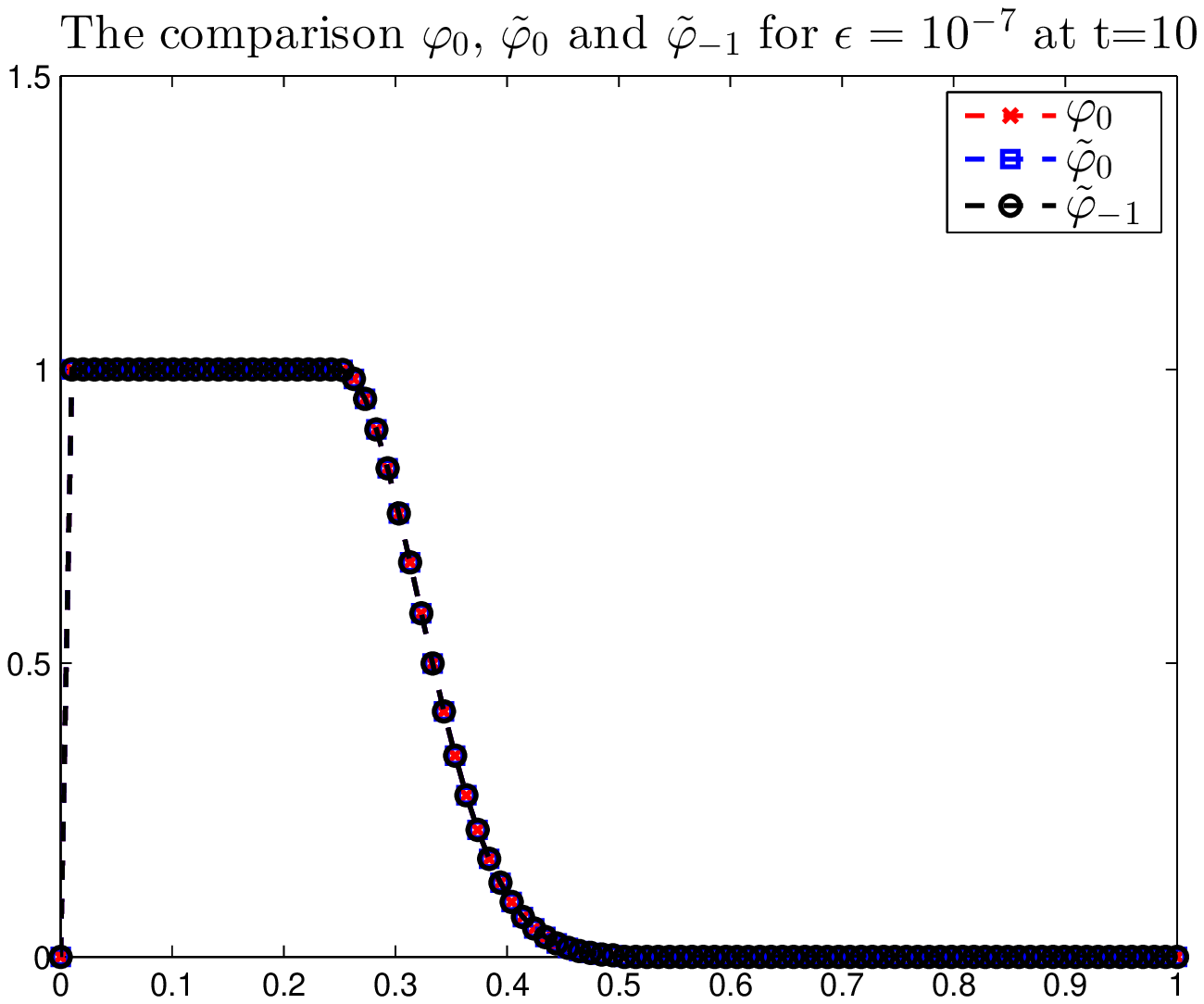}
    \includegraphics[scale=0.47]{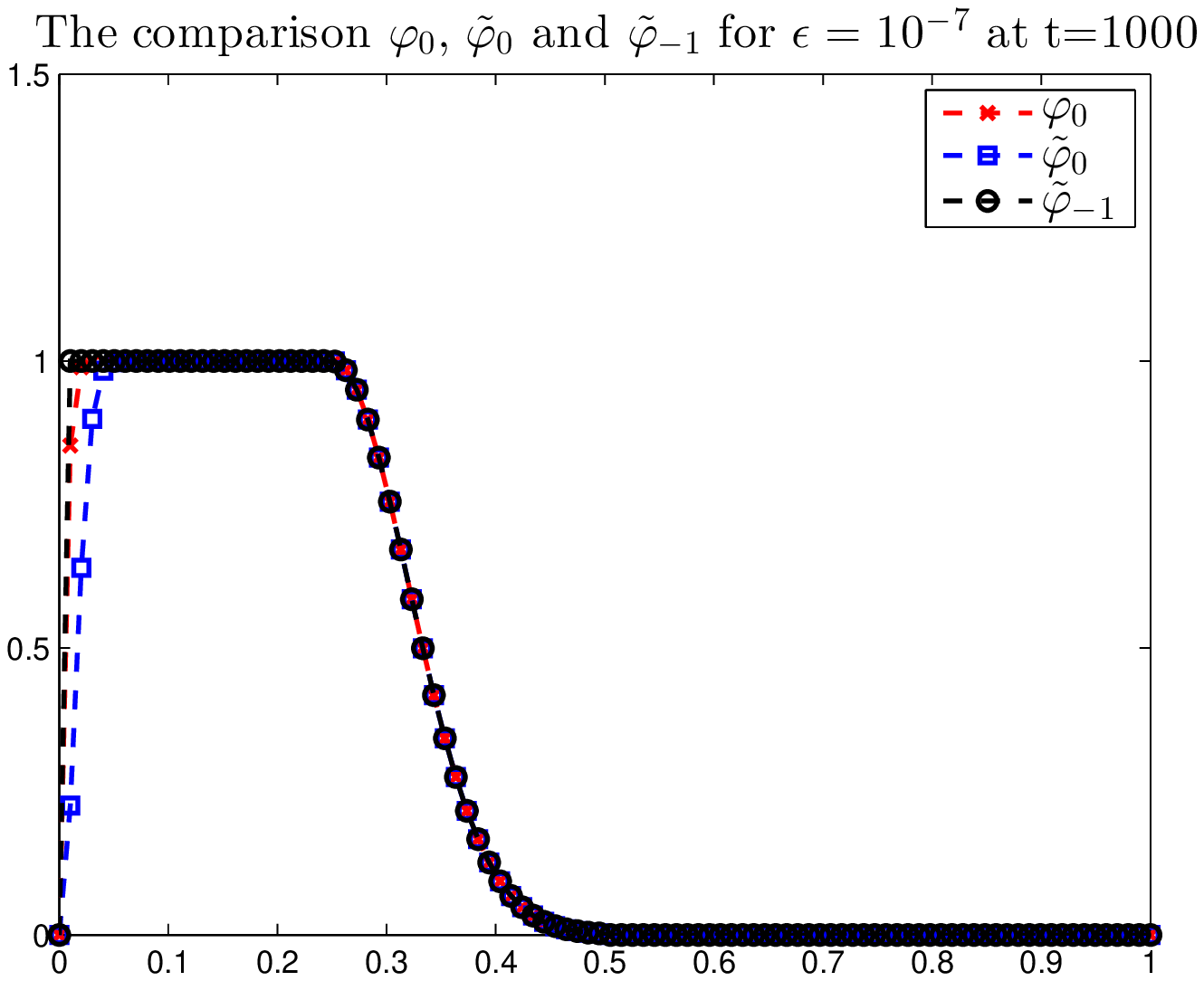}
    \caption{The comparison of $\varphi_0$, $\tilde \varphi_0$ and $\tilde \varphi_{-1}$ for $\e = 10^{-5}$ at $t=10$ and $t=1000$.}
    \label{comp_cor2}
\end{figure}
\begin{remark}
In addition, we use, in the numerical simulations, $\tilde \varphi_{-1}^{lin}$ which is the linearized form of the boundary layer element $\tilde \varphi_{-1}$ such that
\be
	 	 \tilde \varphi_{-1}^{lin} (\xi) = \Big [ 1 - \exp \Big (-\f{\xi^2}{4 \e} \Big)- \Big (1-\exp \Big 					(-\f{\sigma^2}{4 \e} \Big ) \Big) \f{\xi}{\sigma} \Big ] \chi_{[0,\sigma]}(\xi),
\label{eq_bl_element}
\ee
where $\sigma \geq 0$ is chosen in numerical examples below; see e.g. Figure \ref{comp_cor3}. 
Compared with \eqref{app_bl_element2}, we replace the cut-off function $\delta(\xi)$ in \eqref{app_bl_element2} with the linear term in \eqref{eq_bl_element}.
Introducing the linearized boundary layer element $\tilde \varphi_{-1}^{lin}$, the numerical integrations are simpler than with $\tilde \varphi_{-1}$. For more details, see \cite{HJL13} and \cite{HJT13}. 
\end{remark}
\begin{figure}[h]
    \centering
    \includegraphics[scale=0.47]{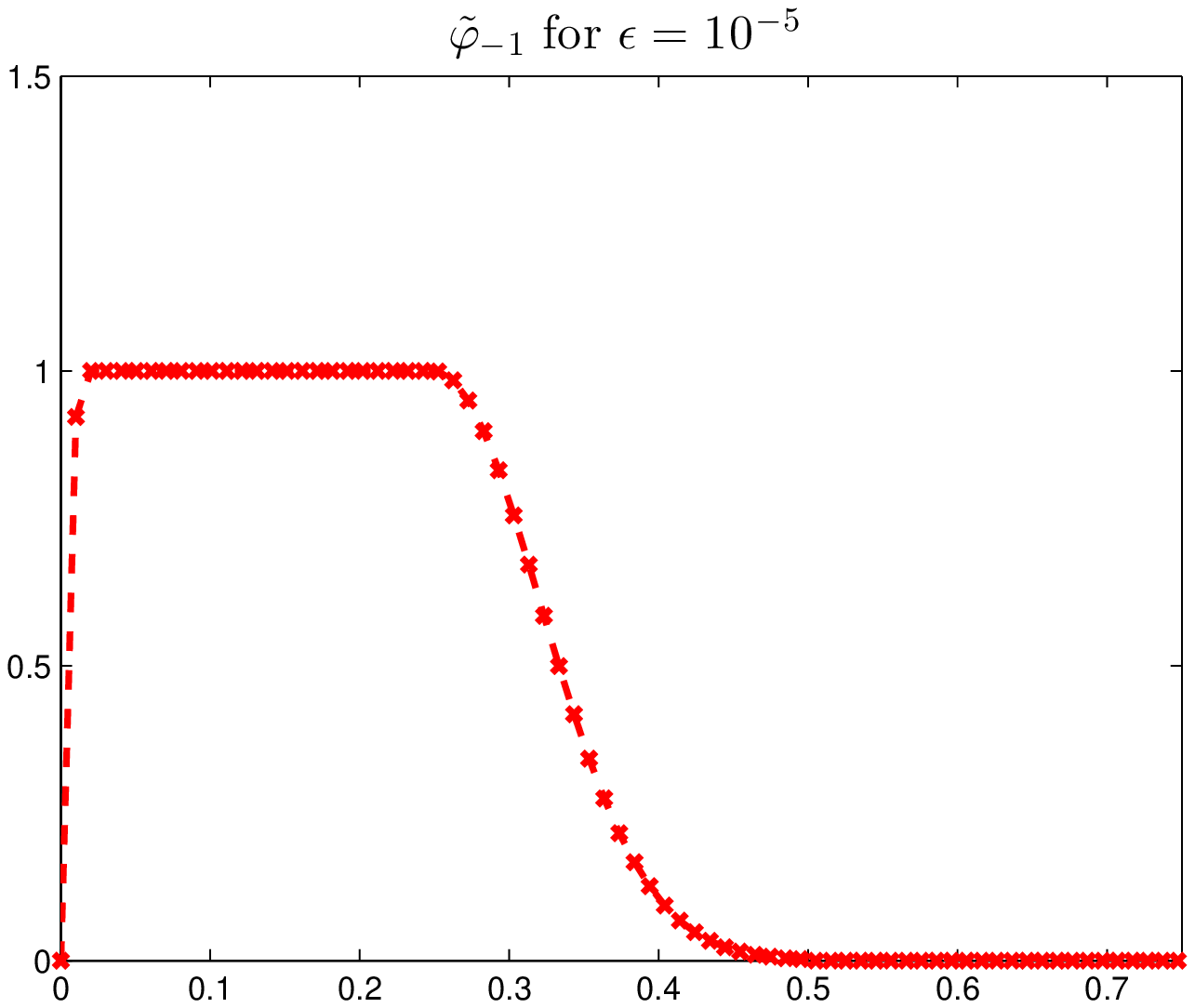}
    \includegraphics[scale=0.47]{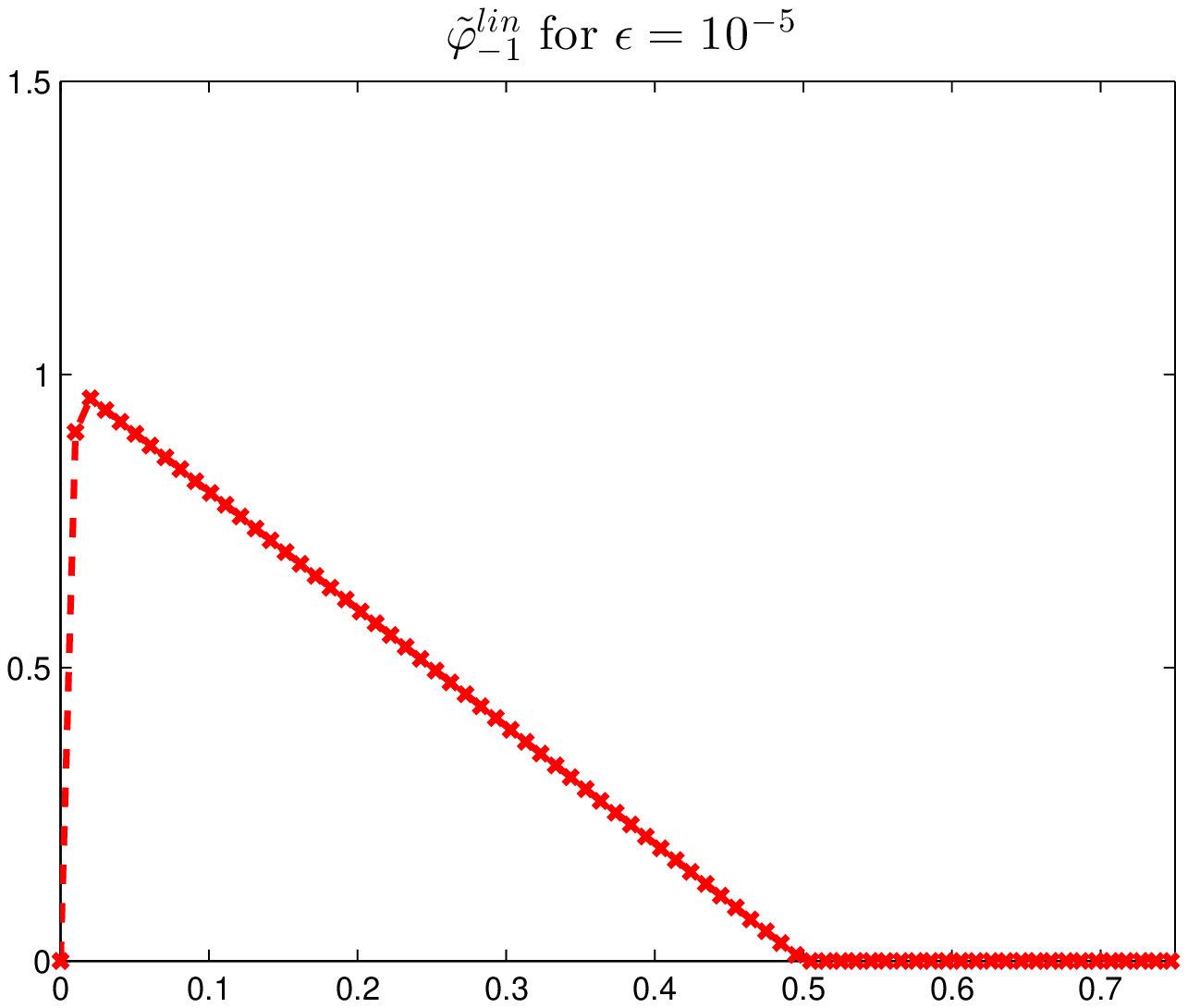}
    \caption{The left figure is the modified corrector $\tilde \varphi_{-1}$, and the right one is the linearized (modified) corrector $\tilde \varphi_{-1}^{lin}$ for $\e = 10^{-5}$.}
    \label{comp_cor3}
\end{figure}
We compute the solution of our problem using a quasi-uniform mesh in place of the adaptive mesh refinement near the boundary layer as in common in the literature \cite{CG12}, \cite{RST08}, \cite{SM05} and \cite{ST03}; see e.g. the triangulation in Figure \ref{fig_mesh}.
For the numerical integrations, we employ the Gauss-Legendre quadrature. Moreover, in the following numerical examples, we apply the implicit Euler method for the time discretizations.

\begin{figure}[h]
    \includegraphics[width=0.7\textwidth]{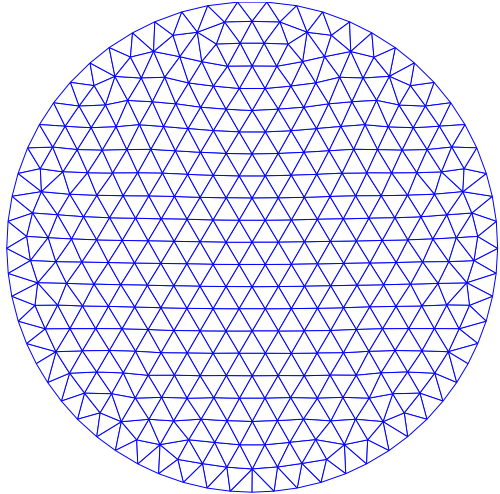}
  \caption{Example of a quasi-uniform grid in a circle.}
\label{fig_mesh}
\end{figure}

\subsection{Simulation 1: One-dimensional example}
We first present a simple one-dimensional example. Indeed, in the previous sections, we mainly focused on the two-dimensional problem, but we can simply reduce our problem to the one-dimensional case. To compare the numerical solutions with the exact solutions, we consider the following equations:
\be
\begin{cases}
	u^\e_t - \e u^\e_{xx} = f(x,t), \quad \text{in  } (0,1) \times (0,T),\\
	u^\e(0,t) =u^\e(1,t) = 0, \quad \text{  }  t \in (0,T),\\
	u^\e(x,0) = u_0(x), \quad \text{  } x \in (0,1). \\
\end{cases}
\label{eq_1d_eq}		
\ee
We choose the exact solution $u^\e(x,t)$ of \eqref{eq_1d_eq} as
\be
		u^\e = t  \Big (1-\exp \Big (-\f{x}{\sqrt{\e}} \Big ) \cos \Big (\f{x}{\sqrt{\e}} \Big ) \Big ) \Big (1-\cos \Big (\f{(1-x)}{\sqrt{\e}} \Big ) \exp \Big (-\f{(1-x)}{\sqrt{\e}} \Big ) \Big ).
\label{eq_1d_exac}
\ee
Hence, $f$ is computed from $\eqref{eq_1d_eq}_1$. In Figure \ref{fig_1d}, we observe that the SFEM method (the solid line) produce oscillations near the boundary, however, with the NFEM method (the dotted line), the boundary layer elements capture the sharp transition near the boundary.\\
Figure \ref{fig_1d_log} shows the rate of convergence of the relative $L^2$ errors for the SFEM method and the NFEM method in log-log scales. We define the relative $L^2$ error
\be
\label{e4.3}
	\f{\|u_{\text{EX}} - u_{\text{N}}\|_{L^2}}{\|u_{\text{EX}}\|_{L^2}},
\ee
where $u_{\text{EX}}$ is the exact solution as in \eqref{eq_1d_exac} and $u_{\text{N}}$ is the numerical solution. According to Figures \ref{fig_1d} and \ref{fig_1d_log}, we observe that the errors from the NFEM method is much smaller than that from the SFEM method.
\begin{figure}[h]
    \includegraphics[width=0.7\textwidth]{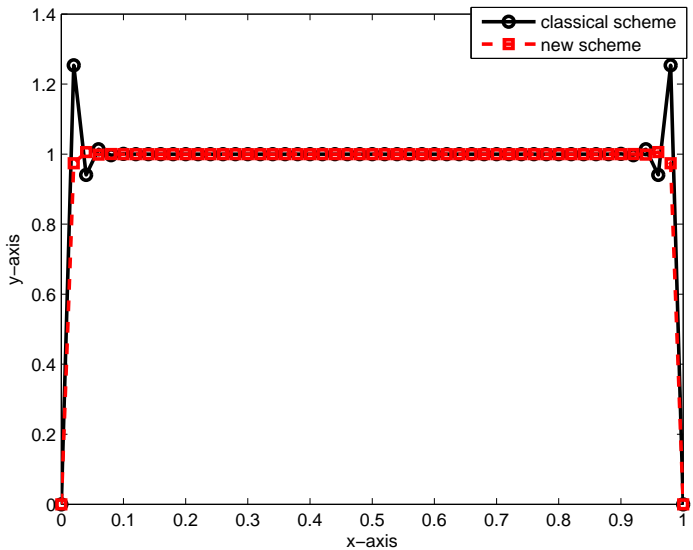} 
  \caption{Solution at $T=1$ of \eqref{eq_1d_eq} where $u^\e$ is as in \eqref{eq_1d_exac} and $\e = 10^{-5}$; the solid line: the classical scheme $u^\e_N$, the dotted line: the new scheme $\bar u^\e_N$. The number of elements is $N=50$ and the size of the time step is $\Delta t = 0.01$.}
\label{fig_1d}
\end{figure}

\begin{figure}[h]
    \includegraphics[width=0.7\textwidth]{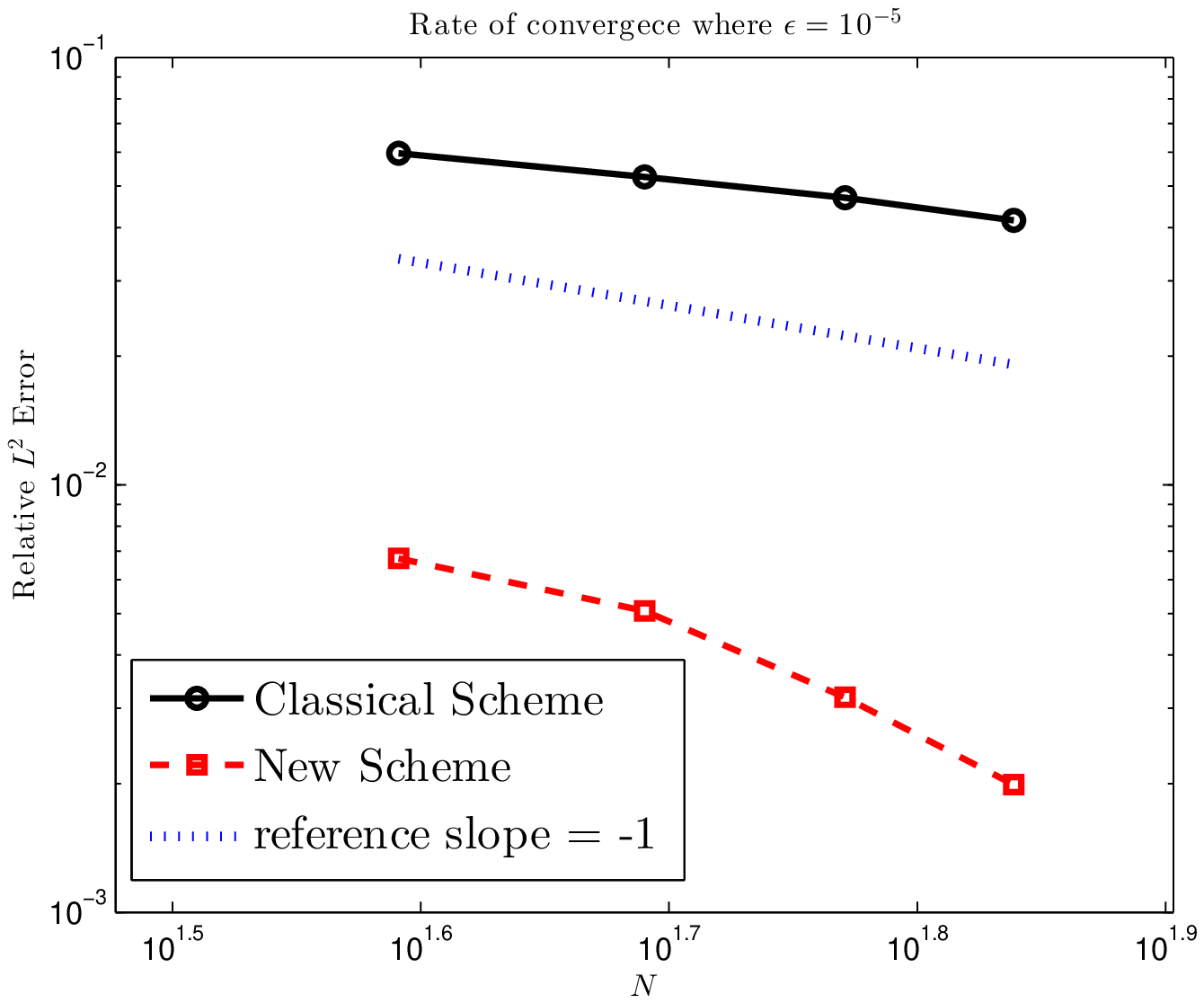}
  \caption{The convergence of the relative $L^2$-errors at time $T=1$ for the solution of \eqref{eq_1d_exac} for two different schemes (the SFEM method and the NFEM method) in log-log scales where $\e=10^{-5}$ and the size of the time step $\Delta t = 0.01$.}
\label{fig_1d_log}
\end{figure}

\subsection{Simulation 2: Two-dimensional example}
For the two-dimensional example, we approximate the following exact solution $u^\e$ of \eqref{eq_main}
\be
	  u^\e = \exp(t) \bigg (1 - \f{I_0  ( \f{r}{\sqrt{\e}}  )}{I_0(\f{1}{\sqrt{\e}})} \bigg ),
\label{eq_2d_exac}	  
\ee
where $I_0(x)$ is the modified Bessel function of the first kind; see e.g. \cite{HJL13}. Then, we can find the corresponding $f=\exp(t)$.
Figure \ref{fig_2d} shows the approximate solutions for different schemes. We find that the numerical solution with the NFEM method (B) is smooth. On the other hand, the solution with the SFEM method (A) displays wild oscillations near the boundary.\\
In Figure \ref{fig_2d_log}, we observe the rate of convergence of the relative $L^2$ errors for the SFEM method and the NFEM method in log-log scales. As we expect, the NFEM method is more accurate than the SFEM method.
\begin{figure}
\begin{center}
$
\begin{array}{lcr}
                (A) \includegraphics[scale = 0.49]{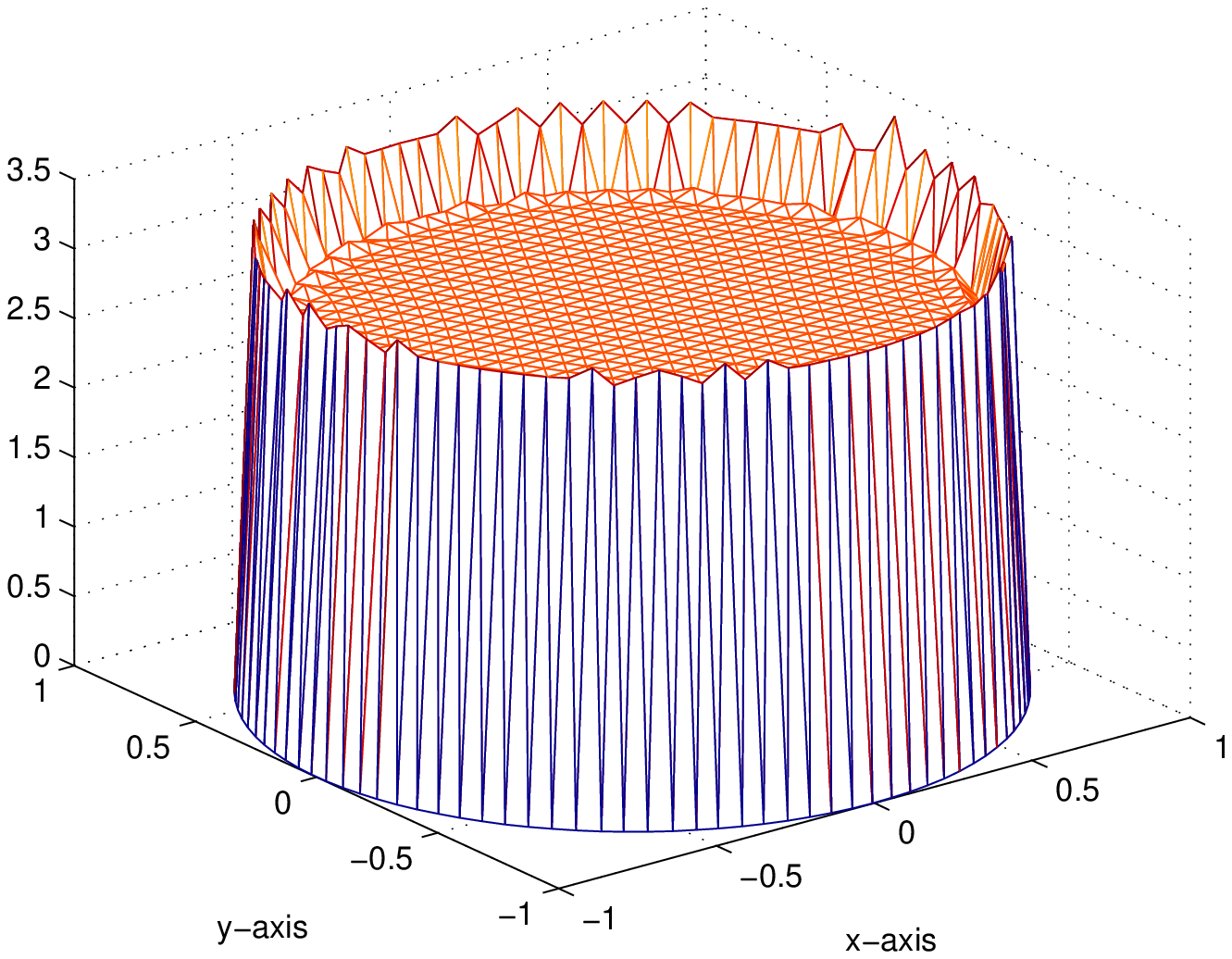}&
                (B) \includegraphics[scale = 0.49]{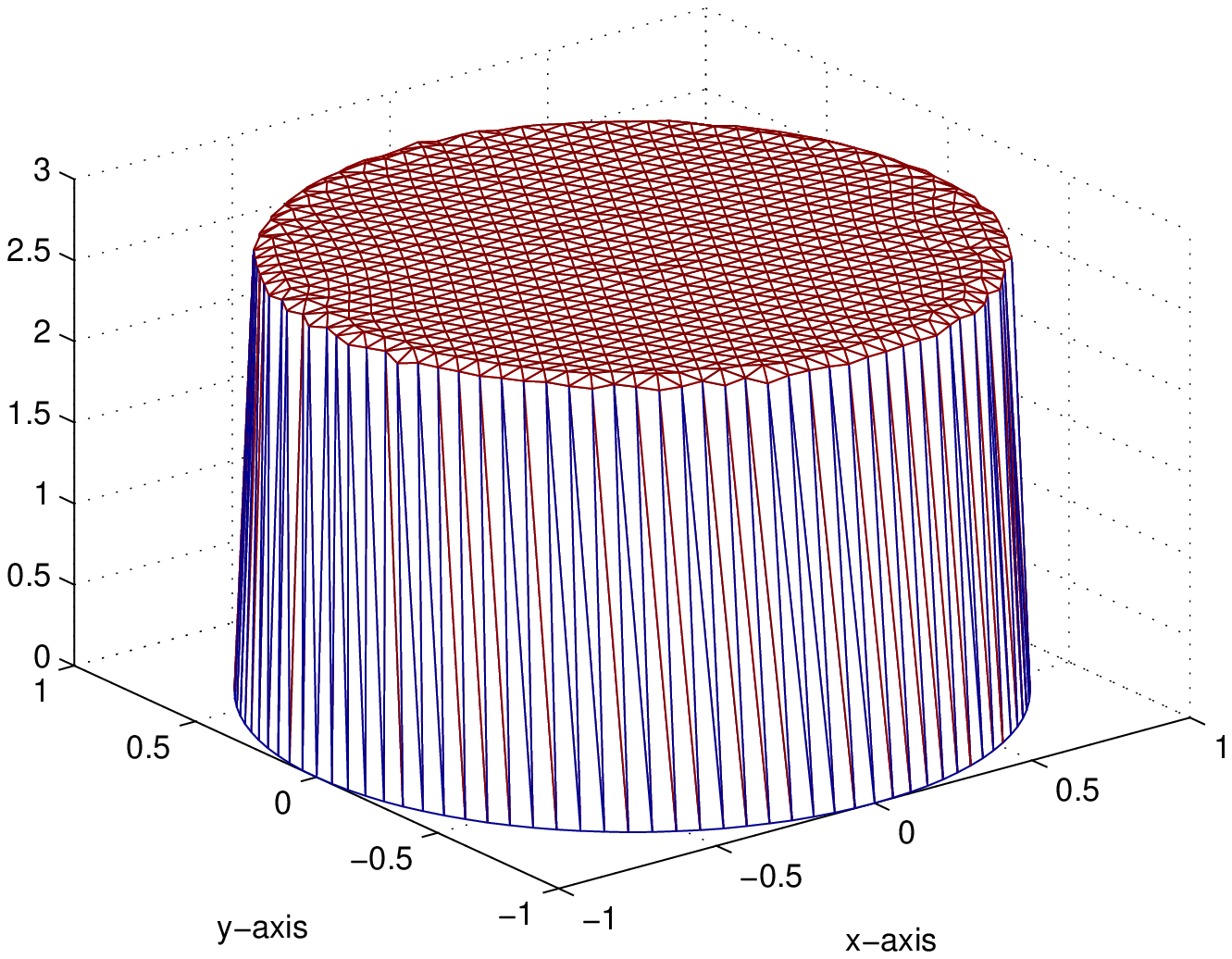}
\end{array}
$
\end{center}
\caption{Solution at $T=1$ of (\ref{eq_main}) where $u^\e$ is as in \eqref{eq_2d_exac} and $\e=10^{-8}$. (A): solution $u_N^{\e}$ with the SFEM method and (B): solution $\bar u_N^\e$ with the NFEM method. The number of elements is $N=1,008$, and the number of boundary nodes is $M=52$. The size of the time step is $\Delta t = 0.01$.}
\label{fig_2d}
\end{figure}

\begin{figure}[ht]
    \includegraphics[width=0.7\textwidth]{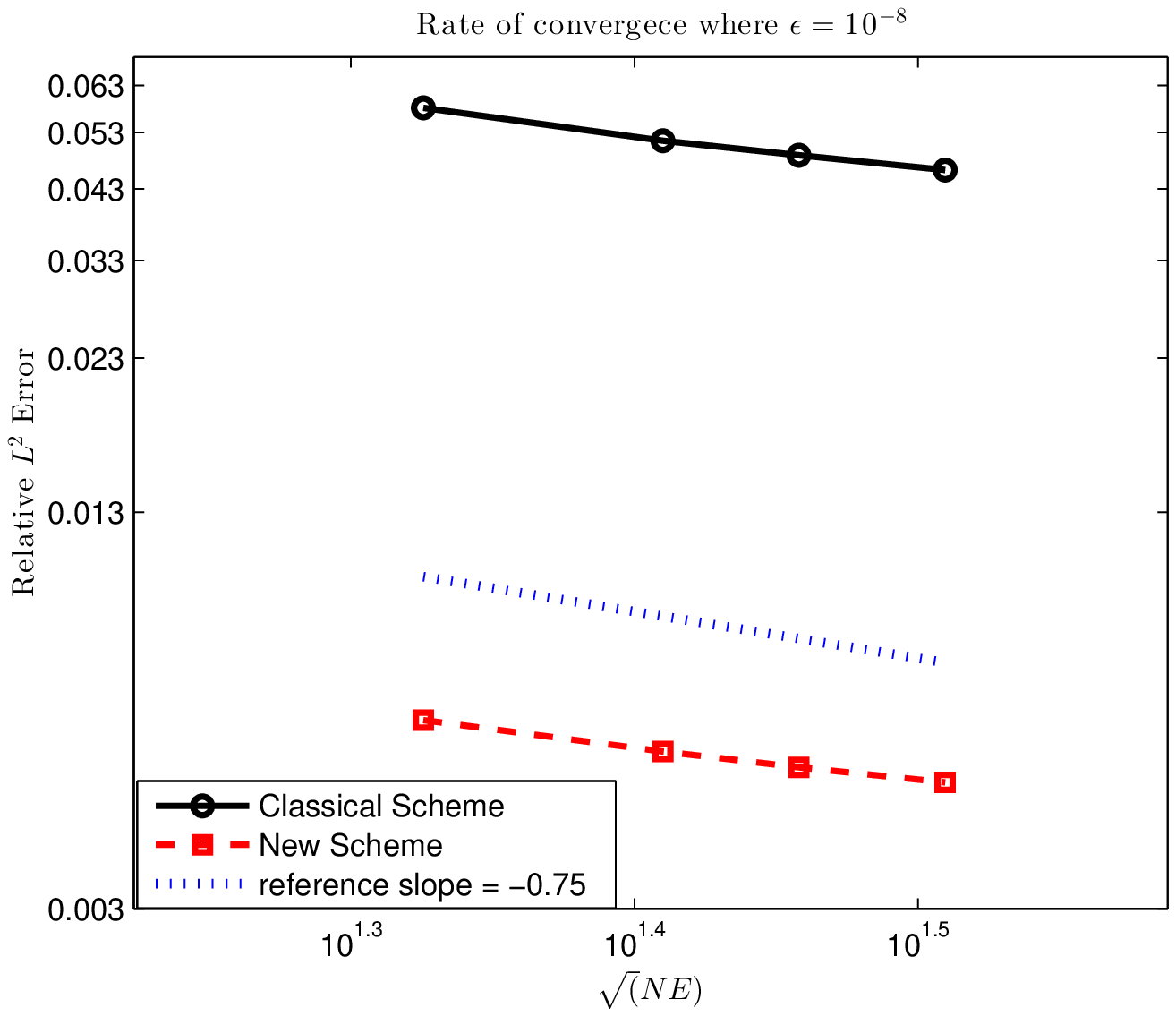}
  \caption{The convergence of the relative $L^2$-errors at time $T=1$ for the solution of \eqref{eq_2d_exac} for two different schemes (the SFEM method and the NFEM method) in log-log scales where $\e=10^{-8}$. The size of the time step is $\Delta t = 0.01$.}
\label{fig_2d_log}
\end{figure}

\section{Conclusion}
The numerical evidence in Section \ref{sec_num} shows that we found an accurate approximate solution for the heat equation with small thermal conductivity. One of the novelties of this article is to compute a non oscillatory numerical solution using a {\it quasi-uniform mesh} in a curved domain.

In the future we intend to study singularly perturbed convection-diffusion equations introducing the convective terms. However, there are two new major difficulties. Firstly, due to the convective terms, the numerical errors propagate into the interior of the domain. Secondly, one should take into account the compatibility conditions near the characteristic points; for more details, see e.g. \cite{JT14} and a forthcoming article \cite{HT14}. One can also extend our results to the Stokes equations (linearized Navier-Stokes equations) as in \cite{GHT11}. However, one should propose a new numerical approach to treat the pressure and the divergence free condition (such as the projection method in \cite{Tem_proj}, \cite{Tem_NSE}, and \cite{CA68}), keeping in mind that the pressure is a global (nonlocal) function of the velocity.

\section*{Acknowledgements}
This work was supported in part by NSF Grants DMS 1206438 and by the Research Fund of Indiana University.


\begin{thebibliography}{10}

\bibitem{AJ67}
Jean-Pierre Aubin.
\newblock Behavior of the error of the approximate solutions of boundary value
  problems for linear elliptic operators by {G}elerkin's and finite difference
  methods.
\newblock {\em Ann. Scuola Norm. Sup. Pisa (3)}, 21:599--637, 1967.

\bibitem{BS94}
Susanne~C. Brenner and L.~Ridgway Scott.
\newblock {\em The mathematical theory of finite element methods}, volume~15 of
  {\em Texts in Applied Mathematics}.
\newblock Springer-Verlag, New York, 1994.

\bibitem{BCGJ07}
B.~Bujanda, C.~Clavero, J.~L. Gracia, and J.~C. Jorge.
\newblock A high order uniformly convergent alternating direction scheme for
  time dependent reaction-diffusion singularly perturbed problems.
\newblock {\em Numer. Math.}, 107(1):1--25, 2007.

\bibitem{Cannon}
John~Rozier Cannon.
\newblock {\em The one-dimensional heat equation}, volume~23 of {\em
  Encyclopedia of Mathematics and its Applications}.
\newblock Addison-Wesley Publishing Company Advanced Book Program, Reading, MA,
  1984.
\newblock With a foreword by Felix E. Browder.

\bibitem{CT02}
Wenfang Cheng and Roger Temam.
\newblock Numerical approximation of one-dimensional stationary diffusion
  equations with boundary layers.
\newblock {\em Comput. \& Fluids}, 31(4-7):453--466, 2002.
\newblock Dedicated to Professor Roger Peyret on the occasion of his 65th
  birthday (Marseille, 1999).

\bibitem{CTW00}
Wenfang Cheng, Roger Temam, and Xiaoming Wang.
\newblock New approximation algorithms for a class of partial differential
  equations displaying boundary layer behavior.
\newblock {\em Methods Appl. Anal.}, 7(2):363--390, 2000.
\newblock Cathleen Morawetz: a great mathematician.

\bibitem{CA68}
Alexandre~Joel Chorin.
\newblock Numerical solution of the {N}avier-{S}tokes equations.
\newblock {\em Math. Comp.}, 22:745--762, 1968.

\bibitem{CP02}
Philippe~G. Ciarlet.
\newblock {\em The finite element method for elliptic problems}, volume~40 of
  {\em Classics in Applied Mathematics}.
\newblock Society for Industrial and Applied Mathematics (SIAM), Philadelphia,
  PA, 2002.
\newblock Reprint of the 1978 original [North-Holland, Amsterdam].

\bibitem{CG12}
C.~Clavero and J.~L. Gracia.
\newblock A high order {HODIE} finite difference scheme for 1{D} parabolic
  singularly perturbed reaction-diffusion problems.
\newblock {\em Appl. Math. Comput.}, 218(9):5067--5080, 2012.

\bibitem{GHT11}
Gung-Min Gie, Makram Hamouda, and Roger Temam.
\newblock Asymptotic analysis of the {S}tokes problem on general bounded
  domains: the case of a characteristic boundary.
\newblock {\em Appl. Anal.}, 89(1):49--66, 2010.

\bibitem{GHT10}
Gung-Min Gie, Makram Hamouda, and Roger Temam.
\newblock Boundary layers in smooth curvilinear domains: parabolic problems.
\newblock {\em Discrete Contin. Dyn. Syst.}, 26(4):1213--1240, 2010.

\bibitem{HJL13}
Youngjoon Hong, Chang-Yeol Jung, and Jacques Laminie.
\newblock Singularly perturbed reaction-diffusion equations in a circle with
  numerical applications.
\newblock {\em International Journal of Computer Mathematics}, to appear.

\bibitem{HT14}
Youngjoon Hong, Chang-Yeol Jung, and Roger Temam.
\newblock Singular perturbation analysis of the time dependent convection-diffusion equations.
\newblock in preparation.

\bibitem{HJT13}
Youngjoon Hong, Chang-Yeol Jung, and Roger Temam.
\newblock On the numerical approximations of stiff convection-diffusion
  equations in a circle.
\newblock {\em Numerische Mathematik}, to appear.

\bibitem{CJ09}
Claes Johnson.
\newblock {\em Numerical solution of partial differential equations by the
  finite element method}.
\newblock Dover Publications Inc., Mineola, NY, 2009.
\newblock Reprint of the 1987 edition.

\bibitem{JT08}
Chang-Yeol Jung.
\newblock Finite elements scheme in enriched subspaces for singularly perturbed
  reaction-diffusion problems on a square domain.
\newblock {\em Asymptot. Anal.}, 57(1-2):41--69, 2008.

\bibitem{JPT11}
Chang-Yeol Jung, Madalina Petcu, and Roger Temam.
\newblock Singular perturbation analysis on a homogeneous ocean circulation
  model.
\newblock {\em Anal. Appl. (Singap.)}, 9(3):275--313, 2011.

\bibitem{JT06}
Chang-Yeol Jung and Roger Temam.
\newblock On parabolic boundary layers for convection-diffusion equations in a
  channel: analysis and numerical applications.
\newblock {\em J. Sci. Comput.}, 28(2-3):361--410, 2006.

\bibitem{JT10}
Chang-Yeol Jung and Roger Temam.
\newblock Finite volume approximation of two-dimensional stiff problems.
\newblock {\em Int. J. Numer. Anal. Model.}, 7(3):462--476, 2010.

\bibitem{JT14}
Chang-Yeol Jung and Roger Temam.
\newblock Boundary layer theory for convection-diffusion equations in a circle.
\newblock {\em Russian Mathematical Survey}, to appear.
\newblock special volume in memory of Mark Vishik.

\bibitem{KS11}
Natalia Kopteva and Simona~Blanca Savescu.
\newblock Pointwise error estimates for a singularly perturbed time-dependent
  semilinear reaction-diffusion problem.
\newblock {\em IMA J. Numer. Anal.}, 31(2):616--639, 2011.

\bibitem{LM10}
Torsten Linss and Niall Madden.
\newblock Analysis of an alternating direction method applied to singularly
  perturbed reaction-diffusion problems.
\newblock {\em Int. J. Numer. Anal. Model.}, 7(3):507--519, 2010.

\bibitem{NJ70}
J.~Nitsche.
\newblock Lineare {S}pline-{F}unktionen und die {M}ethoden von {R}itz f\"ur
  elliptische {R}andwertprobleme.
\newblock {\em Arch. Rational Mech. Anal.}, 36:348--355, 1970.

\bibitem{RST08}
Hans-G{\"o}rg Roos, Martin Stynes, and Lutz Tobiska.
\newblock {\em Robust numerical methods for singularly perturbed differential
  equations}, volume~24 of {\em Springer Series in Computational Mathematics}.
\newblock Springer-Verlag, Berlin, second edition, 2008.
\newblock Convection-diffusion-reaction and flow problems.

\bibitem{RU03}
Hans-G{\"o}rg Roos and Zorica Uzelac.
\newblock The {SDFEM} for a convection-diffusion problem with two small
  parameters.
\newblock {\em Comput. Methods Appl. Math.}, 3(3):443--458 (electronic), 2003.
\newblock Dedicated to John J. H. Miller on the occasion of his 65th birthday.

\bibitem{KS87}
Shagi-Di Shih and R.~Bruce Kellogg.
\newblock Asymptotic analysis of a singular perturbation problem.
\newblock {\em SIAM J. Math. Anal.}, 18(5):1467--1511, 1987.

\bibitem{SF73}
Gilbert Strang and George~J. Fix.
\newblock {\em An analysis of the finite element method}.
\newblock Prentice-Hall Inc., Englewood Cliffs, N. J., 1973.
\newblock Prentice-Hall Series in Automatic Computation.

\bibitem{SM05}
Martin Stynes.
\newblock Steady-state convection-diffusion problems.
\newblock {\em Acta Numer.}, 14:445--508, 2005.

\bibitem{ST03}
Martin Stynes and Lutz Tobiska.
\newblock The {SDFEM} for a convection-diffusion problem with a boundary layer:
  optimal error analysis and enhancement of accuracy.
\newblock {\em SIAM J. Numer. Anal.}, 41(5):1620--1642 (electronic), 2003.


\bibitem{Tem_proj}
Roger Temam.
\newblock Sur l'approximation de la solution des \'{e}quations de Navier-Stokes par la m\'{e}thode des pas fractionnaires II.
\newblock {\em Arch. Rational Mech. Anal.}, 33:377-385, 1969.


\bibitem{Tem_NSE}
Roger Temam.
\newblock {\em Navier-{S}tokes equations}.
\newblock AMS Chelsea Publishing, Providence, RI, 2001.
\newblock Theory and numerical analysis, Reprint of the 1984 edition.

\bibitem{TW95}
Roger Temam and Xiao~Ming Wang.
\newblock Asymptotic analysis of the linearized {N}avier-{S}tokes equations in
  a channel.
\newblock {\em Differential Integral Equations}, 8(7):1591--1618, 1995.

\bibitem{TW97}
Roger Temam and Xiaoming Wang.
\newblock Asymptotic analysis of the linearized {N}avier-{S}tokes equations in
  a general {$2$}{D} domain.
\newblock {\em Asymptot. Anal.}, 14(4):293--321, 1997.

\bibitem{VS13}
Martin Viscor and Martin Stynes.
\newblock A robust finite difference method for a singularly perturbed
  degenerate parabolic problem {II}.
\newblock {\em IMA J. Numer. Anal.}, 33(2):460--480, 2013.

\end{thebibliography}
\end{document}